\documentclass[a4paper,10pt]{article}
\usepackage[a4paper,margin=3cm]{geometry}
\usepackage[T1]{fontenc}
\usepackage{amsfonts,amsthm,amssymb,amsmath, mathdots}
\usepackage{amssymb}
\usepackage[pdftex]{color,graphicx}

\numberwithin{equation}{section}
\title{The limiting distributions of large heavy Wigner and arbitrary random matrices}
\author{Camille Male\footnote{UMPA, ENS de Lyon, 46 all\'ee d'Italie, 69007 Lyon, France. camille.male@ens-lyon.fr}}
\date{}

\newtheorem{Th}{Theorem}[section]

\newtheorem{Def}[Th]{Definition}
\newtheorem{Prop}[Th]{Proposition}
\newtheorem{PropDef}[Th]{Proposition/Definition}

\newtheorem{Lem}[Th]{Lemma}
\newtheorem{Ass}{Assumption}
\newtheorem{Cor}[Th]{Corollary}

\renewcommand\leq\leqslant

\renewcommand\geq\geqslant

\newcommand{\multsetl}{\big\{\!\!\big\{}
\newcommand{\multsetr}{\big\}\!\!\big\}}

\def\Tr{\mathrm{Tr}}

\def\esp{\mathbb E}

\def\etc{,\ldots ,}

\def\toN{^{(N)}}
\def\toNs{^{(N)*}}

\def\limN{\underset{N \rightarrow \infty}\longrightarrow}
\def\Nlim{\underset{N \rightarrow \infty}\lim}

\def\eps{\varepsilon}

\def\Gcyc{\mathcal G_{cyc}\langle \mathbf x, \mathbf x^*\rangle}

\def\eq{\begin{eqnarray*}}
\def\qe{\end{eqnarray*}}
\def\eqa{\begin{eqnarray}}
\def\qea{\end{eqnarray}}

\begin{document}
\maketitle

\begin{center}
\begin{minipage}{12cm}
\begin{center}{\sc abstract:}\end{center}
 {The model of heavy Wigner matrices generalizes the classical ensemble of Wigner matrices: the sub-diagonal entries are independent, identically distributed along to and out of the diagonal, and the moments its entries are of order $\frac 1 N$, where $N$ is the size of the matrices. Adjacency matrices of Erd\"os-Renyi sparse graphs and matrices with properly truncated heavy tailed entries are examples of heavy Wigner matrices. We consider a family $\mathbf X_N$ of independent heavy Wigner matrices and a family $\mathbf Y_N$ of arbitrary random matrices, independent of $\mathbf X_N$, with a technical condition (e.g. the matrices of $\mathbf Y_N$ are deterministic and uniformly bounded in operator norm, or are deterministic diagonal). We characterize the possible limiting joint $^*$-distributions of $(\mathbf X_N, \mathbf Y_N)$ in the sense of free probability. We find that they depend on more than the $^*$-distribution of $\mathbf Y_N$. We use the notion of distributions of traffics and their free product to quantify the information needed on $\mathbf Y_N$ and to infer the limiting distribution of $(\mathbf X_N, \mathbf Y_N)$. We give an explicit combinatorial formula for joint moments of heavy Wigner and independent random matrices. When the matrices of $\mathbf Y_N$ are diagonal, we give recursion formulas for these moments. We deduce a new characterization of the limiting eigenvalues distribution of a single heavy Wigner.}

\end{minipage}
\end{center}
{\bf keywords: $^*$-distribution, asymptotic freeness, Wigner, heavy-tailed random variables, Erd\"os-Renyi graphs.}

\tableofcontents

\section{Introduction}

\subsection{Motivations}

\noindent Let $\mathbf A_N = (A_1 \etc A_p)$ be a family of random $N$ by $N$ matrices with complex entries, whose entries have all their moments. Following random matrix and free probability terminology, we call (mean) {\bf $^*$-distribution} of $\mathbf A_N$ the map
	\eq
		\Phi_{\mathbf A_N} : P \mapsto \esp \Big[ \frac 1 N \Tr \big[ P(\mathbf A_N) \big] \Big],
	\qe
defined on the set of non commutative $^*$-polynomials, i.e. finite complex linear combinations of words in indeterminates $a_1 \etc a_p, a_1^* \etc a_p^*$. When it exists, we call limiting $^*$-distribution of $\mathbf A_N$ the pointwise limit of $\Phi_{\mathbf A_N}$ when $N$ goes to infinity, and say that $\mathbf A_N$ converges in $^*$-distribution.
\\
\\The notion of {\bf asymptotic $^*$-freeness} introduced by Voiculescu gives a rule to compute the limiting $^*$-distribution of a large class of random matrices as their size goes to infinity (see \cite{AGZ,DYK,BBGS11,CC,NIC93,NEA,Shl96} for examples). Recall its definition.

\begin{Def}[Asymptotic $^*$-freeness]~\label{Def:Freeness}
\\Let $\mathbf A_1 \etc \mathbf A_p$ be families of $N$ by $N$ random matrices having a mean limiting joint $^*$-distribution
	\eq
		\Phi : P \mapsto \Nlim \esp\bigg[ \frac 1 N \Tr \Big[ P(\mathbf A_1 \etc \mathbf A_p) \Big] \bigg],
	\qe

\noindent defined on the set of non commutative $^*$-polynomials in indeterminates $\mathbf a_1 \etc \mathbf a_p$. The families $\mathbf A_1 \etc \mathbf A_p$ are asymptotically $^*$-free if and only if for all $^*$-polynomials $P_1, P_2, \dots$, one has $\Phi\big( P_j(\mathbf a_{i_j}) \big)=0$, $i_j\neq i_{j+1}$ for all $j\geq 1$ implies $\Phi\big( P_1(\mathbf a_{i_1}) \dots P_n(\mathbf a_{i_n}) \big) = 0$ for all $n\geq 1$.
\end{Def}

\noindent One of the main examples concerns {\bf independent Wigner and arbitrary random matrices}. Recall that $X_N$ is a Wigner matrix whenever it is Hermitian with independent and centered sub-diagonal entries, such that the diagonal and the extra diagonal entries of $\sqrt N X_N$ are identically distributed according to probability measures, say $\nu$ and $\mu$ respectively, that possess all their moments. Let $\mathbf X_N$ be a family of $N$ by $N$ independent Wigner matrices and $\mathbf Y_N$ a family of $N$ by $N$ arbitrary matrices, possibly random but independent of $\mathbf X_N$. Assume that $\mathbf Y_N$ converges in $^*$-distribution and assume some control on $\mathbf Y_N$ (namely a concentration and a tightness property, see Assumptions 2 and 3 below, e.g. the matrices of $\mathbf Y_N$ are deterministic and uniformly bounded in operator norm). Then Voiculescu's asymptotic freeness theorem \cite{AGZ} states that the limiting $^*$-distribution of $(\mathbf X_N, \mathbf Y_N)$ exists and depends only on the limiting $^*$-distribution of $\mathbf Y_N$ and of the variances of extra diagonal entries of the Wigner matrices (see Corollary \ref{Cor:WignerCase}).
\\
\\This fact reflects a universality phenomenon for eigenvalues statistics of large random matrices, since {\bf the limiting $^*$-distribution of $(\mathbf X_N, \mathbf Y_N)$ does not depend on the details of the law of the entries of Wigner matrices}. Such a result is useful since the convergence in $^*$-distribution of $(\mathbf X_N, \mathbf Y_N)$ implies (and is actually equivalent to) the convergence in moments of the (mean) empirical eigenvalues distribution of any Hermitian matrix $H_N$, obtained as a fixed $^*$-polynomial in $\mathbf X_N$ and $\mathbf Y_N$. Recall that the empirical eigenvalues distribution of an $N$ by $N$ matrix $H_N$ is the probability measure
	\eq
		\mathcal L_{H_N} = \esp\bigg[  \frac 1 N \sum_{i=1}^N \delta_{\lambda_i} \bigg],
	\qe

\noindent where $\lambda_1 \etc \lambda_N$ are the eigenvalues of $H_N$ and $\delta_\lambda$ denotes the Dirac mass in $\lambda$.
\\
\\This theorem of asymptotic $^*$-freeness has an impact in classical probability theory. Consider the case where the families $\mathbf X_N$ and $\mathbf Y_N$ consist only in one Hermitian matrix $X_N$ and $Y_N$ respectively. Then, the convergence in $^*$-distribution of $Y_N$ is the convergence in moments of its empirical eigenvalue distribution toward a probability measure $\pi$. Assume the technical conditions stated in Assumptions 2 and 3, Section \ref{sec:DistrTraff}. Then, Voiculescu's asymptotic freeness theorem gives in particular a characterization of the limiting empirical eigenvalues distribution of the sum $Y_N + X_N$. Following free probability terminology, it is called the {\bf free convolution} of $\pi$ with the semicircle distribution $\sigma_a$, namely the probability measure
	\eq
		\textrm d \sigma_a(t) = \frac 1 {2\pi \sqrt a} \sqrt{ r^2 - t^2 } \mathbf 1_{ |t|^2\leq r^2} \textrm d t, \ \ r=2\sqrt{a} = \Nlim 2 \big( \esp[ N X_N(1,2)^2] \big)^{\frac 1 2}.
	\qe
\noindent In this article we investigate the question of the convergence of $(\mathbf X_N, \mathbf Y_N)$ where Wigner matrices are replaced by matrices of a larger class, with the same structure of independence of its entries.

\begin{Def}[Heavy Wigner matrices]~\label{Def:HeavyWignerMatrices}
\\A random matrix $X_N$ is an $N$ by $N$ heavy Wigner matrix whenever:
\begin{enumerate}
		\item Almost surely, $X_N$ is Hermitian, i.e. $X_N=X_N^*$,
		\item the sub-diagonal entries of $X_N$ are independent and centered,
		\item the diagonal entries of $M_N = \sqrt N X_N$ are distributed according to a measure $\nu_N$ on $\mathbb R$,
		\item the strictly sub-diagonal entries of $M_N$ are distributed according to a measure $\mu_N$ on $\mathbb C$, invariant by complex conjugacy,
		\item $\mu_N$ and $\nu_N$ possess all their moments and for any $k\geq 1$
		\eqa
			\label{Assum:MomentMu}	\frac{ \int |z|^{2k} \textrm d\mu_N(z)}{N^{k-1}} & \limN & a_k,\\
			\label{Assum:MomentNu}	\frac{ \int t^{2k} \textrm d\nu_N(t)}{N^{k-1}} & = & O(1).
		\qea
\end{enumerate}
\end{Def}

\noindent The sequence $(a_k)_{k\geq 1}$ is called the {\bf parameter} of $X_N$. Hence, a Wigner matrix is a heavy Wigner matrix such that, with the notation above, the measures $\nu_N$ and $\mu_N$ do not depend on $N$. In that case, $a_k$ is zero for any $k\geq 2$. Such a parameter is said to be trivial in the following.
\\
\\This matrix model has been introduced independently by two authors. Zakharevich \cite{ZAK} has shown that the empirical eigenvalues distribution of a heavy Wigner matrix converges as $N$ goes to infinity. She has proved that the limiting distribution depends only on the parameter of the matrix. It consists in the semicircular distribution of radius $\sqrt {a_1}$ if the parameter is trivial, and has unbounded support otherwise. Zakharevich has given a combinatorial formula for the moments of this limiting distribution, based on the enumeration of certain rooted trees and she has proved that these moments characterize the measure when $a_k = O(\alpha^k)$ for some $\alpha>0$. Furthermore, Ryan \cite{RYA} has proved that a family $\mathbf X_N$ of independent heavy Wigner matrices has a limiting $^*$-distribution. He has given a combinatorial formula for it, which involves partition generalizing the classical approach for large Wigner matrices based on non crossing pair partitions (see \cite{NS}). In particular, he has proved that Voiculescu's rule of $^*$-freeness does not govern the limiting $^*$-distribution of $\mathbf X_N$ as soon as at least two matrices of the family have a non trivial parameter. Motivated by question from free probability, Benaych-Georges and Cabanal Duvillard \cite{BGCD12} have shown the convergence of the empirical eigenvalues distribution for the generalized Gram matrix $H_{N,M} =X_{N,M} X_{N,M}^*$, where $X_{N,M}$ is an $N$ by $M$ matrix such that:
\begin{itemize}
	\item the ratio $\frac N M$ converges to a positive constant
	\item the entires of  $\sqrt M X_{N,M}$ are independent and identically distributed according to a probability measure $\mu_N$ on $\mathbb C$ whose moments satisfies Ryan-Zakharevich's condition \eqref{Assum:MomentMu}.
\end{itemize}
In this article, we consider a family $\mathbf X_N$ of $N$ by $N$ independent heavy Wigner matrices and a family $\mathbf Y_N$ of $N$ by $N$ matrices, possibly random but independent of $\mathbf X_N$. We characterize the possible limiting $^*$-distribution of $(\mathbf X_N, \mathbf Y_N)$ under suitable assumptions on $\mathbf Y_N$. The most meaningful phenomenon that arises is that {\bf the limiting $^*$-distribution of $(\mathbf X_N, \mathbf Y_N)$ depends on much more than the limiting $^*$-distribution of $\mathbf Y_N$}. We use the notions of distributions of traffics and their free product to specify asymptotic statistics on $\mathbf Y_N$ and then characterizes the $^*$-distribution of $(\mathbf X_N, \mathbf Y_N)$.
\\
\\In particular, if $X_N$ is a heavy Wigner matrix and $Y_N$ a random Hermitian matrix, independent of $X_N$ having a limiting empirical eigenvalues distribution and uniformly bounded in operator norm, the problem of characterizing the limiting eigenvalues distribution of $Y_N + X_N$ is ill-posed. Up to a subsequence, a limiting eigenvalues distribution exists but there can exist many possible limits. For instance, if $Y_N$ is a Wigner matrix independent of $X_N$, then the limiting eigenvalues distribution of $Y_N + X_N$ is the free convolution of Zakharevich's distribution with a semicircular distribution. If $Y_N$ is an arbitrary matrix with fixed limiting eigenvalues distribution (say diagonal and random), then the one of $X_N + Y_N$ depends on the whole distribution of traffics of $Y_N$.

\subsection{Examples of models}

\noindent We point out how the study of heavy Wigner matrices is interesting since this model is related to classical random matrices.

\subsubsection{Matrices with truncated heavy tailed entries}\label{sec:DefLevy}

\noindent We say that a law of a random variable $x$ belongs to the domain of attraction of an $\alpha$ stable law if there exists a function $L: \mathbb R \to \mathbb R$ slowly varying such that
		$$\mathbb P\big( | x | \geq u \big ) = \frac {L(u)}{u^\alpha}, \forall u\in \mathbb R, \ \alpha \in ]0,2[.$$

\noindent A L\'evy matrix $X_N$ with parameter $\alpha$ in $]0,2[$ is a random Hermitian matrix such that: for any $i,j=1\etc N$,
		$$X_N^{(\alpha)}(i,j) = \frac {x_{i,j}}{\sigma_N},$$
where the random variables $(x_{i,j})_{1\leq i\leq j \leq N}$ are independent, identically distributed according to a law that belongs to the domain of attraction of an $\alpha$ stable law for an $\alpha$ in $]0,2[$ and
		$$\sigma_N = \inf \Big \{ \ u \in \mathbb R^+ \ \Big | \ \mathbb P\big( | x_{1,1} | \geq u \big) \leq \frac 1 N \ \Big\}.$$
By the formula for truncated moments of heavy-tailed random variables \cite[Formula (15)]{BG}, for any $B>0$, the random matrix $X_N^B$ whose entries are given by: for any $i,j=1\etc N$,
		$$X_N^{(\alpha, B)}(i,j) =    \frac {x_{i,j}}{B\sigma_N} \mathbf 1_{|x_{i,j}|\leq B \sigma_N},$$
is a heavy Wigner matrix with parameter 
	\eqa
		\Big(  \frac { \alpha}{(2k-\alpha)B^\alpha} \Big)_{k\geq 1}.
	\qea

\noindent The first mathematical results on L\'evy matrices are due to Ben Arous and Guionnet \cite{BG} in 2007, who have shown the convergence of the eigenvalues distribution of a single L\'evy matrix. Belinschi, Dembo et Guionnet \cite{BDG} has studied the limiting spectrum of the sum of a L\'evy matrix and a diagonal matrix, and of a band L\'evy matrices. Moreover, Bordenave, Caputo and Chafa\"i \cite{BCC2} has given an other characterization of the limiting distribution of a L\'evy matrix than one in \cite{BG}. It is based on the local operator convergence of a L\'evy matrix to a certain graph whose entries are labelled by random variables, the Poissonian weighted infinite tree. This is reminiscent with the traffic based approach of this paper for heavy Wigner matrices.
	
\subsubsection{Adjacency matrices of graphs and networks} \label{sec:ExGraphs}

\noindent Let $G_N = (V,E)$ be a simple undirected random graph with $N$ vertices labelled $\{1\etc N\}$. The adjacency matrix of $G_N$ is the matrix 
	\eq
		A_N = \big(\mathbf 1_{\{m,n\} \in  E}\big)_{m,n=1\etc N}.
	\qe

\noindent By simple, we mean without loops nor edges, so that $A_N$ is a symmetric matrix with entries in $\{0,1\}$ and its diagonal elements are zero.
\\
\\{\bf Erd\"os-Renyi sparse graphs:} The only random graph, invariant by re-indexation of its vertices and whose adjacency matrix has independent entries is the Erd\"os-Renyi random graph: it is a random undirected graph with vertices $\{1 \etc N\}$, such that two distinct vertices are linked by an edge with probability $p_N$, independently of the others edges.
\\
\\We consider $p_N$ of the form $\frac \alpha N$ for a fixed $\alpha>0$ and $N$ large, and denote $G_N^{(\alpha)}$ a random Erd\"os-Renyi with that parameter. This is called the {\bf sparse regime}. 
\\
\\Consider the adjacency matrix $A_N^{(\alpha)}$ of $G_N^{(\alpha)}$, and set
	\eq
			X^{(\alpha)}_N =   A^{(\alpha)}_N -  \frac \alpha N J_N  ,
	\qe
where $J_N$ is the $N$ by $N$ matrix with zero on the diagonal and one elsewhere ($X_N$ is made in such a way its entries are centered). Then, $X^{(\alpha)}_N$ is a heavy Wigner matrix with parameter $(\alpha)_{k\geq1}$. It can be observed that $A^{(\alpha)}_N$ and $X^{(\alpha)}_N$ have the same limiting eigenvalues distribution. More generally, one can replace heavy Wigner matrices by adjacency matrices of sparse Erd\"os-Renyi graphs in the results of this article.
\\
\\{\bf Network version:} With $A^{(\alpha)}_N$ as above, denote by $X^{(\alpha)}_N$ the Hermitian random matrix obtained from $A^{(\alpha)}_N$ by replacing its non zero entries by independent, identically distributed random variables. More formally
	\eq
			X^{(\alpha)}_N = A^{(\alpha)}_N \circ M_N,
	\qe 
where $\circ$ denotes the Hadamard (entry-wise) product, and $M_N$ is a random Hermitian matrix, independent of $A^{(\alpha)}_N$, with independent identically distributed entries (up to the Hermitian condition). Assume that the common law of the entries of $M_N$ are distributed according to a measure $\mu$ centered and which possesses all its moments. Then, $X^{(\alpha)}_N$ is a heavy Wigner matrix with parameter 
	\eqa
		\Big(\alpha \times \int |z|^{2k} d\mu(z) \Big)_{k\geq1}.
	\qea

\noindent The spectral theory of weighted graphs is a fields that has been intensively developed in mathematics \cite{CDS95,CHU97}. The analysis of the spectrum of random sparse matrices has been tackled by many authors as Khorunzhy, Shcherbina and Vengerovsky \cite{KSV04}, Ding and Jiang \cite{DJ10} and Shcherbina and Tirozzi \cite{ST10}.
\\
\\Remark that, more generally, if $X_N$ is a heavy Wigner matrix with parameter $(a_k)_{k\geq 1}$ and $M_N$ is as above, independent of $X_N$, then $X_N \circ M_N$ is a heavy Wigner matrix with parameter 
	\eqa
		\big(a _k \times \int |z|^{2k} d\mu(z) \big)_{k\geq1}.
	\qea

\section{Statement of the main results}

\subsection{The limiting distribution of independent heavy Wigner and arbitrary random matrices}

\noindent {\bf Notations:} Let $\mathbf Y_N = (Y_1\toN \etc Y_q\toN)$ of $N$ by $N$ random matrices. Let $U_N$ be a uniform permutation matrix, independent of $\mathbf Y_N$ and denote $Z_j = U_N Y_j\toN U_N^*$ for any $j=1\etc p$. Following Lovasz \cite{Lov09}, we call injective density of $T$ in $\mathbf Y_N$ the quantity
		\eqa \label{eq:InjDens}
			\delta_N^0\big[ T(\mathbf Y_N) \big] =   \esp\bigg[  \prod_{j=1}^K  Z_{\gamma(j)}^{\eps(j)}(k_j,l_j) \ \bigg | \ \mathbf Y_N \bigg],
		\qea
	where $T = \Big\{ \big( k_j,l_j,\gamma(j),\eps(j) \big) \in \{1,2 , \dots \}^2 \times \{1\etc p\} \times \{1,*\}  \Big | j=1\etc K \Big\}$ is a set of indices (fixed as $N$ go to infinity), $M(k,l)$ denotes the $k,l$ entry of a matrix $M$, $M^*$ its conjugate transpose, and $\esp[ \ \cdot \ | \mathbf Y_N ]$ means the conditional expectation with respect to $\mathbf Y_N$.
\\
\\Such a set of indices $T$ is seen as a labelled graph: with the notations above, the set of vertices is $V = \{k_j,l_j \ | \ j=1 \etc K\}$ and the multi-set of edges is $E = \multsetl (k_j,l_j) \ | \ j=1\etc K \multsetr$. Each edge $e$ has a label $x_{\gamma(e)}^{\eps(e)}$. If $T$ is connected, we call it a $^*$-test graph. It is called a cyclic $^*$-graph when the integers $(k_j,l_j)$ can be taken of the form $(k_j,k_{j+1})$, $j=1\etc K$ with $k_{K+1} = k_1$.
\\
\\We state three assumptions on a family $\mathbf Y_N$ of random matrices. 
\\
\\{\bf Assumption 1:} One assumes its {\bf convergence in distribution of traffics} on $\mathcal G_{cyc} \langle \mathbf y, \mathbf y^* \rangle$ (see Section \ref{sec:DistrTraff}), that is, the convergence of the quantities 
	\eq
		 \esp \Big[ \tau^0_N\big[ T(\mathbf Y_N) \big] \Big]:= \esp \bigg[ \frac {(N-1)!}{(N-|V|)!} \delta_N^0\big[ T(\mathbf Y_N) \big] \bigg],
	\qe
for any cyclic $^*$-graph  $T$ with $|V|$ vertices. This mode of convergence extends the convergence in $^*$-distribution and the weak local convergence of graphs \cite{MAL12}.
\\
\\\noindent {\bf Assumption 2:} One assumes a {\bf concentration} hypothesis in this setting: for cyclic $^*$-test graphs $T_1 \etc T_n$,
	\eq
		 \esp \Big[ \tau^0_N\big[ T_1(\mathbf Y_N) \big] \dots \tau^0_N\big[ T_n(\mathbf Y_N) \big] \Big] - \esp \Big[ \tau^0_N\big[ T_1(\mathbf Y_N) \big] \Big] \dots \esp \Big[ \tau^0_N\big[ T_n(\mathbf Y_N) \big] \Big] \limN 0.
	\qe

\noindent {\bf Assumption 3:} One assumes a condition which implies {\bf tightness} in the setting of the main theorem of this paper (Theorem \ref{MainTh}, Section \ref{sec:AsymFreeTh}): for any (non cyclic) $^*$-test graphs $T_1 \etc T_n$,
	\eq
		 \esp \Big[ \tau^0_N\big[ T_1(\mathbf Y_N) \big] \dots \tau^0_N\big[ T_n(\mathbf Y_N) \big] \Big]  = O\Big( N^{ \sum_{i=1}^n(\mathfrak r(T_i)/2-1)} \Big)
	\qe

\noindent where $\mathfrak r(T_i)$ is a positive integer defined by Mingo and Speicher \cite{MS09}, called the number of leaves of its tree of two-edge connected components. It depends on the geometry of $T_i$, see Section \ref{sec:MS}.
\\
\\\noindent We can state the main theorem of this article, where the interest in that it gives an explicit description of the limiting objects (by the traffic freeness, whose definition is recalled in Section \ref{sec:DefFree}).

\begin{Th}[The traffic-asymptotic freeness of $X_1 \etc X_p, \mathbf Y_N$]~\label{MainTh}
\\Let $\mathbf X_N$ be a family independent heavy Wigner matrices, independent of an arbitrary family $\mathbf Y_N$ of random matrices satisfying the three assumptions above. Then, the joint family $(\mathbf X_N, \mathbf Y_N)$ has a limiting $^*$-distribution. Its limit is characterized by the notion of traffic-freeness in the sense of \cite{MAL12}, and depends only on the parameters of the heavy Wigner matrices and of the distribution of traffics of $\mathbf Y_N$ on $\Gcyc$.
\end{Th}
	
	\noindent Theorem \ref{MainTh} is proved in Section \ref{sec:AsymFreeTh}. In the case where the matrices of $\mathbf X_N$ are classical Wigner matrices, we can replace Assumption 1 by the convergence in $^*$-distribution of $\mathbf Y_N$ (see Corollary \ref{Cor:WignerCase}). To the author's knowledge, this improves the usual asymptotic freeness theorem for independent Wigner and arbitrary random matrices.
\\
\\We give examples of families of matrices which satisfy assumptions of Theorem \ref{MainTh}.
\begin{Prop}[Examples of matrix models]\label{Prop:Example}~
	\begin{enumerate}
		\item {\bf Matrices uniformly bounded in operator:} Let $\mathbf Y_N$ be a family of random matrices whose operator norm is almost surely uniformly bounded. Then, up to a subsequence, $\mathbf Y_N$ satisfies Assumption 1 and it always satisfies Assumption 3.
		\item {\bf Diagonal matrices:} Let $\mathbf D_N$ be a family of random diagonal matrices having a limiting $^*$-distribution and satisfying: for any $^*$-polynomials $P_1 \etc P_n$,
			$$\esp\Big[ \frac 1 N \Tr \big[ P_1 (\mathbf D_N) \big]  \dots  \frac 1 N \Tr \big[ P_n (\mathbf D_N) \big] \Big] - \esp\Big[ \frac 1 N \Tr \big[ P_1 (\mathbf D_N) \big] \Big]   \dots \esp\Big[ \frac 1 N \Tr \big[ P_n (\mathbf D_N) \big] \Big] \limN 0.$$
		\noindent Then, $\mathbf D_N$ satisfies the three assumptions of Theorem \ref{MainTh}.
		\item {\bf Adjacency matrices of graphs:} Let $\mathbf Y_N$ be the family of adjacency matrices of a random colored graph $\mathbf G_N$, with uniformly bounded degree, that converges in the sense weak local convergence to a random rooted colored graph. Then, $\mathbf Y_N$ satisfies Assumption 1. Moreover, if it satisfies Assumption 2, then it satisfies Assumption 3.
	 \end{enumerate} 
\end{Prop}	 
	 \noindent As we can use this theorem for $\mathbf Y_N$ containing diagonal matrices of projection, we obtain an analogue of Theorem \ref{MainTh} for covariance matrices (see Proposition \ref{Prop:HeavyCov}). We also obtain the weak convergence of the empirical eigenvalues distribution of Hermitian matrices in independent L\'evy and random matrices (see Proposition \ref{Prop:LevyMatrices}).

\subsection{The limiting distribution of independent, permutation invariant random matrices}

\noindent Theorem \ref{MainTh} is a consequence of a more general theorem, namely the traffic-asymptotic freeness of random matrices on cyclic $^*$-test graphs stated below, which specify a result of \cite{MAL12} in the settings of the three assumptions stated above. It is proved in Section \ref{sec:AsymFreeTh}, once the setting of traffics and their free product have been reminded in Sections \ref{sec:DistrTraff} and \ref{sec:DefFree}.
	
	\begin{Th}[The asymptotic traffic-asymptotic freeness of permutation invariant, independent families of matrices]~
	\\Let $\mathbf Z_N^{(1)} \etc  \mathbf Z_N^{(p)}$ be families of random $N$ by $N$ matrices. Assume that 
	\begin{itemize}
		\item $\mathbf Z_N^{(1)} \etc  \mathbf Z_N^{(p)}$ are {\bf independent},
		\item for any $j=1\etc p$, $\mathbf Z_N^{(j)}$ is {\bf permutation invariant}, except possibly for one $j$ in $\{1\etc p\}$,
		\item $\mathbf Z_N^{(j)}$ satisfies the three assumptions of Theorem [1.2].
	\end{itemize}
	Then, the joint family $(\mathbf Z_N^{(1)} \etc  \mathbf Z_N^{(p)})$ has a limiting distribution of traffics on $\mathbf G_{cyc}\langle \mathbf z, \mathbf z^* \rangle$, which is the {\bf traffic-free product} of the limiting distributions of $\mathbf Z_N^{(1)} \etc  \mathbf Z_N^{(p)}$ on $\Gcyc$. In particular, it has a limiting $^*$-distribution which depends only on the limiting distribution of traffics of $\mathbf Z_N^{(1)} \etc \mathbf Z_N^{(p)}$ separately.
	\end{Th}

\noindent We prove that a single heavy Wigner matrix satisfies these assumptions. Hence, consider independent heavy Wigner matrices $\mathbf X_N= (X_1\toN  \etc X_p\toN)$ and random matrices satisfying the three assumptions. We get that the families  $(X_1\toN) \etc (X_p\toN), \mathbf Y_N$ are asymptotic traffic-free (see Section \ref{sec:DefFree}) as $N$ goes to infinity. The rule of traffic-freeness gives an explicit way to compute $^*$-moments in $(\mathbf X_N, \mathbf Y_N)$. It is based on the computation of the limiting distribution of traffics given in the following proposition.

\begin{Prop}[The limiting distribution of traffics of a heavy Wigner matrix]\label{Prop:DistrHeavy}~
\\Let $X_N$ be a heavy Wigner matrix with parameter $(a_k)_{k\geq 1}$. Then, $X_N$ has a mean limiting distribution of traffics on $\mathcal G \langle x \rangle$ given by: for any cyclic $^*$-test graph $T$, 
	\eq
		\esp\Big[ \tau_N^0\big[T(X_N) \big] \Big]  \limN \left\{ \begin{array}{cc} \prod_{k\geq 1} a_k^{q_k} &  \textrm{ if } T \textrm{ is a fat tree of type } (q_k)_{k\geq 1},\\
		0 & \textrm{ otherwise.} \end{array} \right.
	\qe

\noindent A fat tree is a $^*$-test graph which becomes a tree if we forget the multiplicity and the orientation of the edges. A fat tree is of type $(q_k)_{k\geq 1}$ if it has $q_k$ undirected edges of multiplicity $2k$. See Figure \ref{fig:FatTrees.pdf}.
\end{Prop}

\begin{figure}[!h]
\begin{center}
 \includegraphics[height=34mm]{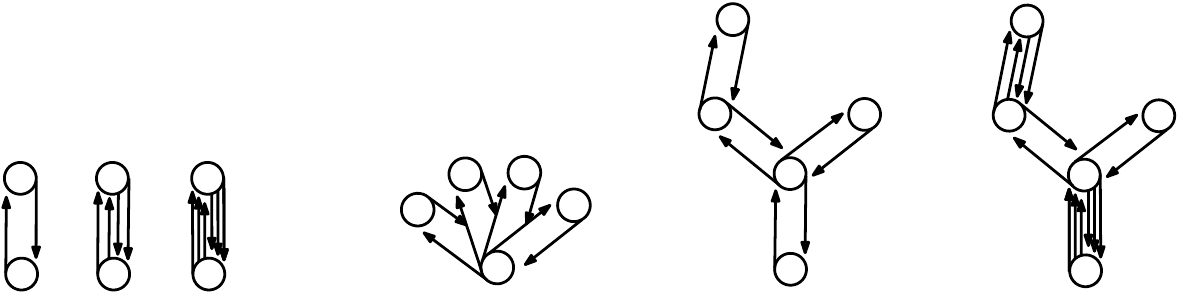}\caption{\label{fig:FatTrees.pdf}Left: the first fat trees with two vertices. From left to right, they are of type $(1,0,0,\dots), (0,1,0,0,\dots)$ and $(0,0,1,0,0,\dots)$. Other examples: from left to right, two fat trees of type $(4,0,0, \dots)$ and one of type $(2,1,1,0,0,\dots)$.}
\end{center}
\end{figure}

\subsection{Limiting joint $^*$-moments of $(\mathbf X_N, \mathbf Y_N)$}

\noindent Based on the results of Section 2, we give a combinatorial description of the limiting $^*$-distribution of $(\mathbf X_N, \mathbf Y_N)$. This approach is different that Ryan's one \cite{RYA} by the so-called clickable partitions. It can be considered as a dual version (see Figure \ref{FatColorCycle}). This generalizes the description \cite{GUI} for non-heavy Wigner matrices by rooted, oriented trees and the description \cite{KSV04} for a single adjacency matrix of a Erd\"os-Renyi sparse weighted graph (see Section \ref{sec:ExGraphs}) by minimal walks on such trees.
\\
\\We first precise the vocabulary we use in order to avoid ambiguities. A tree is a graph with no cycles. We call rooted tree in the complex plane a undirected tree possessing one marked vertex (called its root) embedded in the non negative half plane of $\mathbb C^2$ by planting its root at the origin. A directed edge of such a tree refers to a pair of adjacent vertices. A cycle on a tree is a sequence of directed edges of the form $\big( (v_1, v_2) \etc (v_{L-1}, v_L), (v_L,v_1) \big)$. With this notation, $L$ is called the length of the cycle and the directed edge $(v_n,v_{n+1})$ is called the $n$-th step of $c$ ($n=1\etc L$ with indices modulo $L$).
\\
\\From now, we fix a $^*$-polynomial $P = x_{\gamma(1)} P_1(\mathbf y) \dots x_{\gamma(L)} P_L(\mathbf y)$, where $L\geq 1$, $\gamma: \{1\etc L\} \to \{1\etc p\}$, and $P_1Ê\etc P_L$ are polynomials. We set 	
	\eq
		\Phi(P) = \Nlim \esp \big[ \frac 1 N \Tr \big[ P(\mathbf X_N, \mathbf Y_N) \big] \Big],
	\qe

\noindent and give a combinatorial description of this quantity. By linearity and traciality, this characterizes the limiting $^*$-distribution of $(\mathbf X_N, \mathbf Y_N)$.

\begin{Def}[Colored, minimal cycles on trees]~
\\We set $\mathcal  L^{(\gamma)}$ the set of all couples $(G,c)$, where $G$ is a rooted tree in the complex plane with less than $\lfloor \frac L 2 \rfloor$ edges, and $c$ is a cycle on $G$ with the following properties:
\begin{itemize}
	\item $c$ starts at the root of the tree. When it visits a new vertex, it visits the leftmost one. It visits all the vertices of $G$ and has $L$ steps.
	\item By convention, we say that the $n$-th step of $c$ is of color $\gamma(n)$, $n=1\etc L$. Then, $c$ must visit each edge of $T$ with a single color.
\end{itemize}
\end{Def}

\noindent We have drawn some examples of minimal cycles on trees (in one color) in Figure \ref{fig:RunningOnTheTree}. From now we fix $(G,c)$ in $\mathcal  L^{(\gamma)}$ and define separately weights associated to heavy Wigner and other matrices.

\begin{figure}[!h]
\begin{center}
\includegraphics[width=140mm]{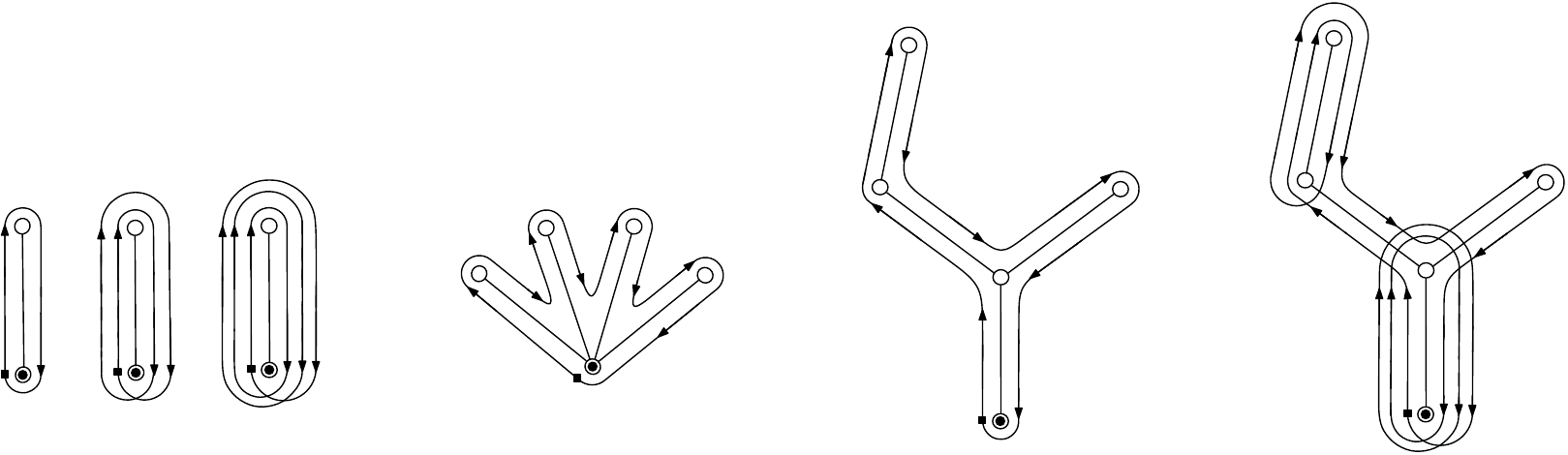}
\caption{\label{fig:RunningOnTheTree} Examples of minimal cycles on trees (in one color). The root of the tree is marked with a black dotted point. The start of the cycle is announced by a black square. We figure out the direction of the cycle each time it approaches a vertex.}
\end{center}
\end{figure}

\begin{Def}[Heavy Wigner weights]~\label{def:HWWeight}
\\ For any edge $e$ of the tree, we denote by $j(e)$ the color in $\{1\etc p\}$ of steps on this edge, and by $2k(e)$ the number of times the cycle visits it. We set 
	\eqa \label{eq:HWWeight}
		 \omega_{HW}(G, c) = \prod_{  e \textrm{ edge of } G}  a_{j(e),k(e)},
	\qea
where for any $j=1\etc p$, $(a_{j,k})_{k\geq 1}$ stands for the parameter of $X_j\toN$. 
\end{Def}

\begin{Def}[Traffic weights]~\label{def:DWeight}
\\Write the cycle $c = (e_1 \etc e_L)$, where $e_j$ is a directed edge of the tree. For any vertex $v$ of the tree $G$, we define a $^*$-test graph $T_v$ in the variables $P_1(\mathbf y) \etc P_L(\mathbf y)$. The vertices of $T_v$ are the incident edges of $G$ in $v$. If the $n$-th step of $c$ is incident at $v$, then we get an edge between the undirected edges corresponding to $e_n$ and $e_{n+1}$ (with the convention $e_{L+1}=e_1$) which is labelled $P_n(\mathbf y)$. We set
	\eqa
		 \omega_{TR}(G, c) = \prod_{ v \textrm{ vertex of } G}  \tau[T_v],
	\qea
where for any $^*$-test graph $T$, with set of vertices $V$, set of edges $E$ and whose edge $e$ in $E$ is labelled $x_{\gamma(e)}^{\eps(e)}$, 
	\eq
		\tau[T] := \Nlim \frac 1 N \sum_{\phi : V \to \{1\etc N\} } \prod_{e \in E} P_{\gamma(e)}^{\eps(e)}(\mathbf Y_N) \big( \phi(e) \big),
	\qe
(see Section \ref{sec:DistrTraff})
\end{Def}

\noindent See Figures \ref{fig:NodesToTestGraphs3} and \ref{fig:NodesToTestGraphs} for illustrations and concrete procedure to compute these weights.

\begin{figure}[!h]
\parbox{.5cm}{~}
\parbox{9cm}{ \includegraphics[width=80mm]{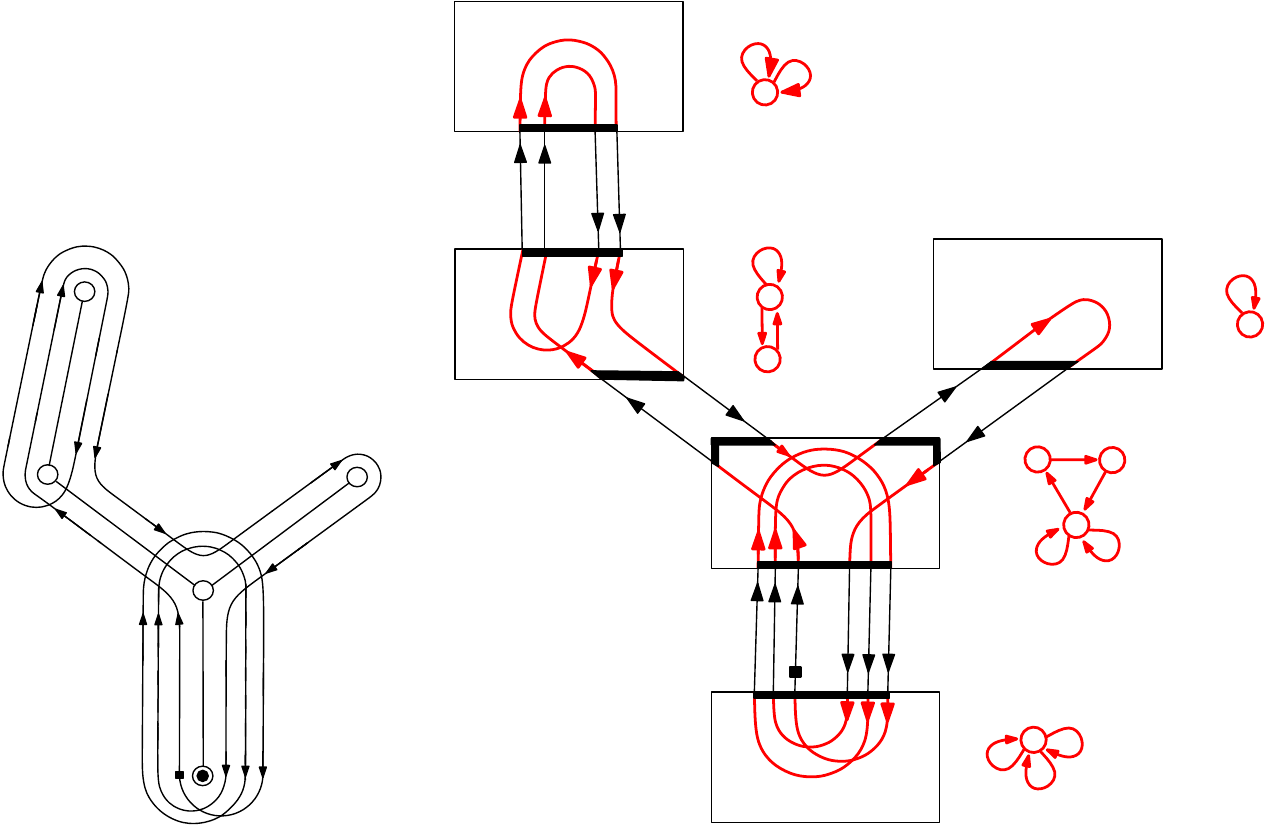}}
\parbox{5cm}{ \caption{\label{fig:NodesToTestGraphs3} Left: A colored, minimal cycles on a tree. Right: We extract the features of the cycle by drawing boxes around each vertex and forgetting the initial tree. The edges out of the boxes give the weight $\omega_{HW}$: we count the number and the color of visits of each edges and use formula \eqref{eq:HWWeight}. On the right of each box, we have represented the associated $^*$-test graph in variables $P_1(\mathbf y) \etc P_L(\mathbf y)$. The product of their trace of $^*$-test graph (see Section \ref{sec:DistrTraff}) gives the weight $\omega_{TR}$.} }
\end{figure}

\begin{figure}[!h]
\begin{center}
\includegraphics[width=140mm]{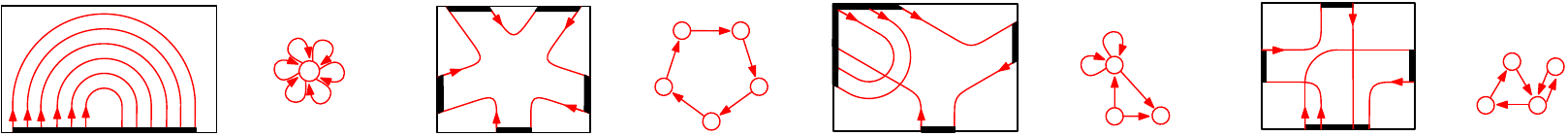}
\caption{\label{fig:NodesToTestGraphs}Some boxes that can be obtained by the process of Figure \ref{fig:NodesToTestGraphs3} and the graphs obtained from them.}
\end{center}
\end{figure}

\begin{Prop}[Joint moments in heavy Wigner and deterministic matrices]~\label{Prop:ComputMoments}
\\For any polynomial $P$ of the form
	$P = x_{\gamma(1)} P_1(\mathbf y) \dots x_{\gamma(L)} P_L(\mathbf y)$,
one has
	\eqa
		\Phi(P) = \sum_{(G,c) \in \mathcal L^{(\gamma)}}   \omega_{HW}(G, c) \times  \omega_{TR}(G, c).
	\qea
\end{Prop}

\noindent We deduce from Proposition \ref{Prop:ComputMoments} the following simple facts.

\begin{Cor}[Basic properties of limiting variables]~\label{Cor:BasicProp}
\begin{enumerate}
	 	\item For any $j=1\etc p$, $\Phi ( x_j ) = 0$ and $\Phi ( x_j^2 ) = a_{j,1}$.
	 	\item For any $L\geq 1$, any $\gamma : \{1\etc L\} \to \{1 \etc p\}$ and any $^*$-polynomials $P_1(\mathbf y) \etc P_L(\mathbf y)$, the quantity $\Phi \big( x_{\gamma(1)}P_1(\mathbf y) \dots  x_{\gamma(L)} P_L(\mathbf y) \big) $ vanishes as soon as the number of occurrence of one variable is odd.
		\item for any integers $n_1\etc n_L\geq 0$, any distinct indices $i_1Ê\etc i_L$ in $\{1\etc p\}$ and any $^*$-polynomial $P$, one has $\Phi \big( x_{i_1}^{n_1} \dots x_{i_L}^{n_L} P(\mathbf y) \big) =  \Phi ( x_{i_1}^{n_1} ) \dots \Phi ( x_{i_L}^{n_L} ) \Phi\big( P(\mathbf y) \big)$.
	\end{enumerate}
\end{Cor}

\subsection{Applications}

\subsubsection{The non asymptotic $^*$-freeness of heavy Wigner and independent arbitrary matrices}

\noindent The notions of $^*$-freeness and traffics freeness are different \cite{MAL12}. In particular, by Definition \ref{Def:Freeness}, if the matrices $X_N$ and $Y_N$ are asymptotically $^*$-free, then we get by definition \ref{Def:Freeness}
	\eq
		f(x,y) := \Phi  \Big (  \big(x^2 - \Phi( x^2) \big) \big( y^2 - \Phi( y^2) \big)   \big( x^2 - \Phi(x^2) \big) \big( y^{*2} - \Phi(y^{*2}) \big)  \Big) =0.
	\qe

\noindent The following Proposition extends the result of Ryan \cite{RYA} which states that independent heavy Wigner matrices with non trivial parameters are not asymptotically $^*$-free. We define the bilinear form
	\eqa\label{eq:Phi2}
		\Phi^{(2)} : (P_1,P_2) \mapsto \Nlim \esp \bigg[ \frac 1 N\Tr \big[  P_1(\mathbf X_N, \mathbf Y_N)  \circ P_2(\mathbf X_N, \mathbf Y_N) \big]  \bigg],
	\qea

\noindent defined for non commutative $^*$-polynomials $P_1$ and $P_2$, where $\circ$ stands for the Hadamard (entry-wise) product. This quantity is well defined since it can be written as the limit of the trace of cyclic $^*$-test graphs in $\mathbf X_N, \mathbf Y_N$ when $P_1$ and $P_2$ are monic $^*$-monomials (see \cite{MAL12}).

\begin{Prop}[The non asymptotic freeness of heavy Wigner and deterministic matrices]~\label{Prop:NonAsymptFree}
\\Let $X_N$ be a heavy Wigner matrix with parameters $(a_{k})_{k\geq 1}$. Let $Y_N$ be an arbitrary random matrix, independent of $X_N$ and satisfying the assumptions of Theorem \ref{MainTh}. Denote by $\Phi$ their joint limiting $^*$-distribution given by Theorem \ref{MainTh} and $\Phi^{(2)}$ the bilinear form given by (\ref{eq:Phi2}). Then, one has 
	\eq
		f(x,y) = a_2 \times g(y), \ \ g(y) =   \Phi^{(2)}( y^2 , y^{*2} ) - \big| \Phi(y^2) \big|^2.
	\qe
\end{Prop}

\noindent Remark that $a_2$ is nonzero as soon as the parameter of $X_N$ is not trivial (Proposition \ref{prop:NecessaryCondition}). Here are examples of matrices $Y_N$ as in the theorem such that $(X_N, Y_N)$ are not $^*$-free. 

\begin{Prop}[Example of random matrices non asymptotically $^*$-free from $\mathbf X_N$]~\label{Prop:NonAsymptFree2}
\begin{enumerate}
	\item If $Y_N$ is a heavy Wigner matrix with parameter $(b_k)_{k\geq 1}$, then $ g(y) = b_2$.
	\item If $Y_N$ is a diagonal matrix having a limiting $^*$-distribution, then 
		\eq	
			 g(y) = \Nlim \esp\Big[ \frac 1 N \Tr [ Y_N^2 Y_N^{*2} ] -  \big| \frac 1 N \Tr [Y_N^2] \big|^2 \Big].
		\qe
	Hence, $(X_N,Y_N)$ are not asymptotically $^*$-free as soon as the limiting eigenvalues distribution of $Y_N$ is not a Dirac mass.
\end{enumerate}
\end{Prop}

\subsubsection{A system of Schwinger-Dyson equation for the limiting distribution of independent heavy Wigner and diagonal matrices}

\noindent We prove recurrent relations for the joint moments in $(\mathbf X_N, \mathbf Y_N)$ in the case where the matrices of $\mathbf Y_N$ are diagonal and it satisfies Assumptions 1 and 2. The philosophy of the proof is the same as for the recurrence relation on sparse, weighted, random graphs \cite{KSV04}. A difference with our approach is that these equations involve only moments in the entries of the matrices rather than purely combinatorial quantities. 
\\
\\For any integer $K\geq 1$, we set the $K$-linear forms
	\eq
		\Phi^{(K)} : & (P_1 \etc P_K) \mapsto & \Nlim \esp \bigg[ \frac 1 N \Tr \big[ P_1(\mathbf X_N, \mathbf Y_N) \circ \dots \circ P_K(\mathbf X_N, \mathbf Y_N) \big] \bigg],
	\qe

\noindent defined for $^*$-polynomials $P_1 \etc P_K$ in variables $\mathbf x, \mathbf y$, where $\circ$ denotes the Hadamard (entry-wise) product of matrices. The maps $(\Phi^{(K)})_{K\geq 1}$ are simple examples of statistics of distributions of traffics that are not defined for $^*$-distributions.

\begin{Th}[A Schwinger-Dyson system of equations]\label{th:SchwingerDyson}~
\\For any $j=1\etc p$, we set $(a_{j,k})_{k\geq 1}$ the parameter of the matrix $X_j\toN$. Then, the family of linear forms $\big( \Phi^{(K)}\big)_{K\geq 1}$ satisfies the following equations. For any integer $K\geq1$, any monomials $P_1\etc P_K$ and any $j=1\etc p$, one has
\begin{eqnarray}
		 \Phi^{(K)}(  x_jP_1,P_2 \etc P_K)  & = & \sum_{k\geq 1} a_{j,k}\sum_{\substack{ s_1+\dots + s_K=k \\ s_1\geq 1, \ s_2\etc s_K\geq 0}} \ \sum_{\mathbf L, \mathbf R} \Phi^{(k)}\big ( \mathbf L ) \Phi^{(k+K-1)}\big ( \mathbf R ),\label{TraceFormEq2}
\end{eqnarray}
where the last sum is over all the families of monomials
\begin{eqnarray*}
		\mathbf L & = & (L_1^{(1)}\etc L_{s_1}^{(1)} \etc L_1^{(K)}\etc L_{s_K}^{(K)}),\\
		\mathbf R & = & (R_1^{(1)}\etc R_{s_1}^{(1)},R_0^{(2)}\etc R_{s_2}^{(2)} \etc R_0^{(K)}\etc R_{s_K}^{(K)}
		),
\end{eqnarray*}
such that
\begin{eqnarray*}
		x_jP_1& = & (x_jL_1^{(1)}x_j)R_1^{(1)} \dots  (x_jL_{s_1}^{(1)}x_j)R_{s_1}^{(1)}\\
		P_k  & = & R_0^{(k)}  (x_jL_1^{(k)}x_j)R_1^{(k)} \dots  (x_jL_{s_k}^{(k)}x_j)R_{s_k}^{(k)}, \ k=2 \etc K.
\end{eqnarray*}
\end{Th}

\noindent This gives a characterization of the limiting $^*$-distribution of $(\mathbf X_N , \mathbf Y_N)$ and a way to compute joint $^*$-moments.

\subsubsection{The spectrum of a single heavy Wigner matrix} We deduce from Schwinger-Dyson system of equations a characterization of the spectrum of a single heavy Wigner matrix $X_N$. Denote by $(a_k)_{k\geq 1}$ its parameter. For any $K\geq 1$, we set the for formal power series $G^{\lambda}  (K)$ in $\frac 1 \lambda$
		$$G^{\lambda}  (K) := \frac 1 {\lambda^K} \sum_{n\geq 0} \frac 1 {\lambda^n} \sum_{\substack{ n_1+ \dots n_K=n \\ n_1\etc n_K\geq 1}} \Phi^{(K)}(x^{n_1}\etc x^{n_K}).$$
This quantity is simply a formal analogue of
		$$\Phi^{(K)}\big( (\lambda-x)^{-1}\etc (\lambda-x)^{-1} \big),$$

\noindent as the terms $ (\lambda-x)^{-1}$ are expended into formal power series $ (\lambda-x)^{-1} = \frac 1 \lambda \sum_{n\geq 0} \frac 1 {\lambda^n} x^n$, and the order between $\Phi^{(K)}$ and these sums are interchanged. In particular, $G^{\lambda}  (1)$ is a formal analogue for the limit of
		$$\esp \Big[ \frac 1 N \Tr \big[  (\lambda-X_N)^{-1} \big] \Big],$$
		called the Stieltjes transform of $X_N$

\begin{Prop}\label{prop:SeriesFormelles}~
\\For any $K\geq 1$, we have the equality between formal power series in $\frac 1 \lambda$
	\eqa
		G^{\lambda}  (K) & = & \frac 1 \lambda \bigg(G^\lambda(K-1)+ \sum_{k\geq 1} a_k \binom{K+k-2}{K-1}G^{\lambda}  (k) G^{\lambda}  (k+K-1) \bigg).
	\qea
These equations characterize the sequence $\big( G^\lambda(K)  \big)_{K\geq 1}$ among the set of formal power series $\big( \tilde G^\lambda(K)  \big)_{K\geq 1}$ such that for any $K\geq 1$, the valence of $\tilde G^\lambda(K) $ is larger than $K$.
\end{Prop}

\noindent Remind that if we assume that there exist some constants $\alpha, \beta>0$ such that for any $k\geq 1$, $a_k \leq \alpha k^\beta$, then by a result of Zakharevich \cite{ZAK}, the limiting eigenvalues distribution of $X_N$ is characterized by its moments.

\newpage
\section[The convergence of heavy Wigner and arbitrary matrices]{The convergence of heavy Wigner and arbitrary matrices}

\noindent We first remind some definitions and results in \cite{MAL12}. We recall and reformulate the two first assumptions of our main theorem.

\subsection{The distribution of traffics of matrices}\label{sec:DistrTraff}

\subsubsection{Statement of Assumptions 1 and 2}

\begin{Def}[$^*$-test graphs]~\label{def:TestGraphs}
\begin{enumerate}
	\item A $^*$-test graph in variables $\mathbf x = (x_1 \etc x_p)$ is a finite, connected, oriented graph (with possibly multiple edges and cycles) whose edges are labelled by variables $x_1 \etc x_p, x_1^* \etc x_p^*$. Formally, it consists in a triplet $T=(V,E,\gamma, \eps)$ where $(V,E)$ is a graph and $\gamma, \eps$ are maps from $E$ to $\{1\etc p\}$ and $\{1,*\}$ respectively, in such a way an edge $e$ in $E$ is labelled $x_{\gamma(e)}^{\eps(e)}$.
	\item A $^*$-test graph is said to be cyclic whenever there exists a cycle on its graph visiting each edge once in the sense of its orientation. We denote $\Gcyc$ the sets of cyclic $^*$-test graphs in indeterminates $\mathbf x$ (we keep these notation even though the notation for indeterminates can change).
\end{enumerate}
\end{Def}

\noindent Let $\mathbf Y_N = (Y_1 \etc Y_p)$ be a family of $N$ by $N$ random matrices. For any $^*$-test graph $T$, we call the trace of $T$ in $\mathbf Y_N$ the quantity
\begin{eqnarray} \label{eq:DefTraceGraphs}
		\textrm{\bf Traffic moments: }\tau_N\big[ T(\mathbf Y_N)\big] & := & \frac 1 N \sum_{\phi: V\to \{1\etc N\}} \prod_{e \in E} Y_{\gamma(e)}^{ \eps(e)}
		\big(\phi(e)\big), \ \ \ \ \ \ ~
\end{eqnarray}
	where 
	\begin{itemize}
		\item for any directed edge $e=(v_1,v_2)$, we have set $\phi(e)=(\phi(v_1),\phi(v_2))$,
		\item  and for any $N$ by $N$ matrix $M$ and any integers $n,m$ in $\{1\etc N\}$, the complex number $M(n,m)$ is the entry $(n,m)$ of $M$.
		\item $M^*$ is the conjugate transpose of the matrix $M$.
	\end{itemize}

\noindent With the same notations, we call the injective trace of $T$ in $\mathbf Y_N$ the quantity
\begin{eqnarray}
		\textrm{\bf Traffic cumulants: }\tau^0_N\big[ T(\mathbf Y_N)\big] & := & \frac 1 N \sum_{\substack{ \phi: V\to \{1\etc N\} \\ \textrm{injective} }} \prod_{e \in E} Y_{\gamma(e)}^{ \eps(e)}
		\big(\phi(e)\big). \ \ \ \ \ \ ~
\end{eqnarray}

\noindent The definition of $\tau_N^0$ is consistent with the definition of the introduction. Indeed, let $\sigma_N$ be a random permutation of $\{1\etc N\}$ independent of $\mathbf Y_N$ and $U_N$ the random permutation matrix associated to $\sigma_N$. Then, for any $^*$-test graph $T$,
	\eq
		  \tau^0_N\big[ T(\mathbf Y_N)\big] 
			& = & \frac 1 N \sum_{\substack{ \phi: V\to \{1\etc N\} \\ \textrm{injective} }} 
					\esp \bigg[ \prod_{e \in E} Y_{\gamma(e)}^{ \eps(e)} \big( \sigma_N \circ \phi(e)\big) \ \bigg| \ \mathbf Y_N \bigg] 
 \\
			& = & \frac 1 N \sum_{\substack{ \phi: V\to \{1\etc N\} \\ \textrm{injective} }} 
					\esp \bigg[ \prod_{e \in E} \Big( U_N Y_{\gamma(e)}^{ \eps(e)} U_N^* \Big)\big(   \phi(e)\big) \ \bigg| \ \mathbf Y_N \bigg] 
 \\
			& = & \frac{ (N-1)!}{ (N-|V|)!} \delta_N^0\big[T(\mathbf Y_N)\big],
	\qe

\noindent where $\delta_N^0$ is the injective density defined in \eqref{eq:InjDens}.
\\
\\One can compute the trace of $^*$-test graphs in terms of injective traces and vice versa. This is a consequence of Formula (\ref{eq:TauTauZero}) stated below and of simple facts on posets (see {\rm \cite{NS}}).

\begin{PropDef}[Trace and injective trace of $^*$-test graphs]\label{prop:InjectiveTrace}~
\\Let $T$ be a $^*$-test graph whose set of vertices is denoted by $V$. Let $\pi$ be a partition of $V$. We denote by $T^\pi$ the $^*$-test graph obtained by identification of vertices that belong to a same block of $\pi$, see Figure \ref{SimpleCycledef.pdf}.
\\
\\{\bf 1. Matrix setting:} For any family $\mathbf Y_N$ of $N\times N$ matrices and any $^*$-test graph $T$,
\begin{equation}\label{eq:TauTauZero}
	\tau_N\big[ T(\mathbf Y_N) \big] = \sum_{  \pi \in \mathcal P(V) } \tau_N^0\big[ T^\pi(\mathbf Y_N) \big],
\end{equation}
where $\mathcal P(V)$ is the set of partitions of the set of vertices $V$ of $T$. Hence, one has
\begin{equation}\label{eq:TauZeroTau}
	\tau^0_N\big[T(\mathbf Y_N) \big] = \sum_{ \pi \in \mathcal P(V)  } \tau_N\big[ T^\pi(\mathbf Y_N) \big] \times \mu_{V}(\pi),
\end{equation}
where $\mu_{V}$ is the M\"obius function of the finite poset $\mathcal P(V)$ (see {\rm \cite{NS}}).
\\
\\{\bf 2. General setting:} Given a map $\tau : \Gcyc \to \mathbb C$, we define its injective version by: for all cyclic $^*$-test graph $T$ with set of vertices denoted by $V$,
	\eqa \label{eq:DefInj}
		\tau^0[T] = \sum_{ \pi \in \mathcal P(V)  } \tau[T] \times \mu_{V}(\pi),
	\qea
so that
	\eqa \label{eq:InjStand}
		\tau[T] = \sum_{ \pi \in \mathcal P(V)  } \tau^0[T].
	\qea

\noindent This definition is then consistent with the definition of $\tau_N^0$.
\end{PropDef}
		
	\begin{figure}[!h]
		\parbox{7.5cm}{ \includegraphics[height=25mm]{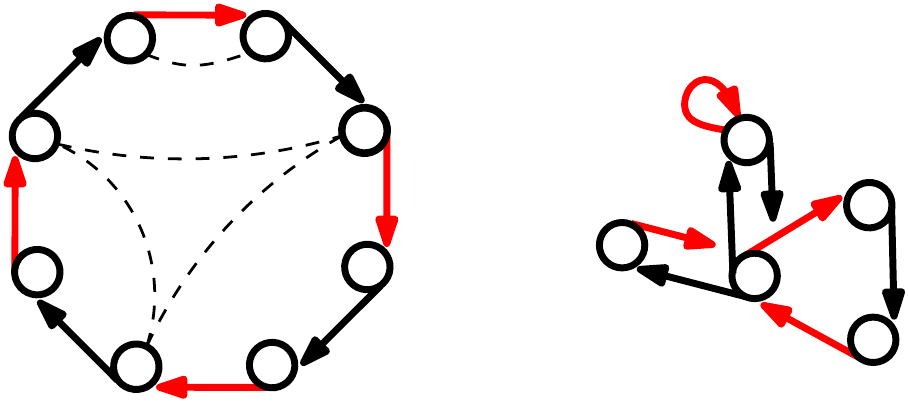}}
		\parbox{6cm}{ \caption{Left: a $^*$-test graph $T$ (labels are replace by colors for simplicity) and a partition $\pi$ of its vertices, represented by dashed lines joining vertices in a same block. Right: the $^*$-test graph $T_{\pi}$.} 
			\label{SimpleCycledef.pdf}
		}
	\end{figure}
	
\noindent The two first assumptions we have stated in the introduction can be reformulated in terms of the non-injective trace.

\begin{Ass}[Convergence in distribution of traffics]~\label{ConvTraff}
\\The family of $N$ by $N$ random matrices $\mathbf Y_N$ has a (mean) limiting distribution of traffics $\tau$ on $\Gcyc$, that is: the entries of $\mathbf Y_N$ have finite moments of any order and for any cyclic $^*$-test graph $T$, 
	\eqa
		\tau[T] : = \Nlim \esp\Big[ \tau_N \big[ T(\mathbf Y_N) \big] \Big] \textrm{ exists.}
	\qea
	
\noindent Equivalently, by Proposition/Definition \ref{prop:InjectiveTrace}, one can replace the trace by the injective one.
\end{Ass}

\begin{Ass}[Concentration]~\label{Ass2}
\\The family of $N$ by $N$ random matrices $\mathbf Y_N$ satisfies for any cyclic $^*$-test graphs $T_1\etc T_n$, 
	\eqa\label{eq:ConvEsp}
		\esp\Big[ \tau_N \big[ T_1(\mathbf Y_N) \big] \dots  \tau_N \big[ T_n(\mathbf Y_N) \big]  \Big]    \limN   \tau[T_1] \dots \tau[T_n].
	\qea
	
\noindent Equivalently, by Proposition/Definition \ref{prop:InjectiveTrace}, one can replace the trace by the injective one.
\end{Ass}

\noindent It is important to have in mind that the $^*$-distribution of a family $\mathbf Y_N$ can be written explicitly in terms of its distribution of traffics $T \mapsto \tau_N[T]$.

\begin{Prop}[Properties of the convergence in distribution of traffics on $\Gcyc$]~
\begin{enumerate}
	\item If almost surely the matrices of $\mathbf Y_N$ are uniformly bounded in operator norm, then up to a subsequence $\mathbf Y_N$ has a limiting distribution of traffics.
	\item The convergence in distribution of traffics on $\Gcyc$ of a family $\mathbf Y_N$ of matrices implies its convergence in $^*$-distribution.
	\item It also implies the convergence in distribution of traffics on $\Gcyc$ of any family of $^*$-polynomials in $\mathbf Y_N$.
\end{enumerate}
\end{Prop}

\noindent The first point is a consequence of a result of Mingo and Speicher recalled in Theorem \ref{Th:MS} below. If $\mathbf Y_N$ has a limiting distribution of traffics $\tau$, the third point of the Proposition gives a sense of $\tau[T]$, where $T$ is a $^*$-test graph whose edges are labelled by polynomials in the indeterminates.

\begin{proof}[Proof of 2.] 
Let $P$ be a $^*$-monomial of the form
	\eq
		P = x_{\gamma(1)}^{\eps(1)} \dots x_{\gamma(K)}^{\eps(K)},
	\qe

\noindent where $K\geq 0$, $\gamma: \{1\etc K\} \to \{1\etc p\}$ and $\eps: \{1 \etc K\} \to \{1,*\}$. Let $T_P $ be the cyclic $^*$-test graph whose vertices are $1 \etc K$ and whose edges are $(1,2) \etc (K-1,K), (K,1)$, the edge $(k,k+1)$ being labelled $x_{\gamma(k)}^{\eps(k)}$ (with indices modulo $K$). Then, one has
	\eq
		\frac 1 N \Tr \big[ P(\mathbf Y_N) \big] = \tau_N \big[ T_P(\mathbf Y_N) \big],
	\qe

\noindent so that the $^*$-distribution of $\mathbf Y_N$ on $^*$-monomials is the restriction of the distribution of traffics of $\mathbf Y_N$ on a subset of cyclic $^*$-test graphs.
\end{proof}

\begin{proof}[Proof of 3.] Let $P_1 \etc P_q$ be $^*$-monomials and define $\mathbf Z_N = \big(P_1(\mathbf Y_N) \etc P_q(\mathbf Y_N) \big)$. For any $n=1\etc q$, we write
		$$P_n = x_{\gamma_{n}(1)}^{\eps_{n}(1)}   \dots   x_{\gamma_{n}(K_n)}^{\eps_{k}(K_n)},$$	
where $K_n\geq 0$, $\gamma_n: \{1\etc K_n\} \to \{1\etc p\}$ and $\eps_{n}:\{1\etc K_n\} \to\{1,*\}$.
\\
\\Then, for any $^*$-test graph $T$, one has $\tau_N\big[ T(\mathbf Z_N) \big] = \tau_N\big[ \tilde T(\mathbf Y_N) \big]$, where $\tilde T$ is the cyclic $^*$-test graph obtained from $T$ by replacing for any $n=1\etc q$ and $\eps$ in $\{1,*\}$, the edges labelled $x_n^\eps$ by a consecutive sequence of edges $e_1 \etc e_{K_n}$, where $e_j$ is labelled $(x_{\gamma_n(j)}^{\eps_n(j)})^\eps$, $j=1\etc K_n$ (with the convention $(x_j^*)^*=x_j$). The convergence of $\mathbf Y_N$ implies the convergence of $\mathbf Z_N$. We get the expected result by multi-linearity.
\end{proof}

\subsubsection{Statement of Assumption 3}
\label{sec:MS}

\noindent Recall the definitions of Mingo and Speicher \cite{MS09}.

\begin{Def}[Tree of two-edges connected components of a $^*$-test graph]~
\begin{enumerate}
	\item A cutting edge of a $^*$-test graph is an edge whose removal would result into disconnected subgraphs. A two-edge connected $^*$-test graph is a $^*$-test graph without cutting edges. A two-edge connected component of a $^*$-test graph is a subgraph which is two-edge connected and cannot be enlarged to a bigger two-edge connected subgraph.
	\item Let $T$ a $^*$-test graph. Its tree of two-edge connected components $\mathfrak T(T)$ is the directed graph defined as follow. The vertices of $\mathfrak T(T)$ consists in the two-edge connected components of $T$. Two distinct vertices of $\mathfrak T(T)$ are connected by an edge if there is a cutting edge between vertices from the two corresponding components in $T$. Hence $\mathfrak T(T)$ is always a tree, i.e. a connected graph without cycles.
	\item A tree is trivial if it consists in only one vertex. A leaf of a non-trivial tree is a vertex which meets only one edge. By convention, we say that the trivial tree has two leafs.
	\item For any $^*$-test graph $T$, we denote by $\mathfrak r(T)$ the number of leaves of $\mathfrak T(T)$.
\end{enumerate}
\end{Def}

\noindent Mingo and Speicher have proved in \cite{MS09} an optimal estimate, reformulated in the language of $^*$-test graphs as follow.

\begin{Th}[Sharp bounds for the trace of test graphs in matrices, {\cite{MS09}}]\label{Th:MS}~
\\Let $T$ be a $^*$-test graph. Let $\mathfrak T(T)$ be its tree of two-edge connected components and denote by $\mathfrak r(T)$ its number of leaves. Then, for any family $\mathbf Y_N$ of $N$ by $N$ matrices, 
\begin{equation}
		\Big| \tau_N \big[ T(\mathbf Y_N) \big] \Big| \leq N ^{ {\mathfrak r(T)}/2-1} \prod_{e\in E} \| Y_{\gamma(e)}\toN\|,
\end{equation}
where $\| \cdot \|$ stands for the operator norm. Moreover, there exists matrices for which this bound is reached.
\end{Th}

\noindent Their result sheds light on the quantity $\mathfrak r (T)$ which turns out to plays an important role in the asymptotic traffic-freeness theorem we prove in this article.

\begin{Ass}[Control of growth]~\label{TechAss}
\\The family of $N$ by $N$ random matrices $\mathbf Y_N$ satisfies: for any (non cyclic) $^*$-test graphs $T_1 \etc T_n$, there exists a constant $C$ such that
	\eqa
	 \esp\Big[ \tau_N \big[ T_1(\mathbf Y_N) \big] \dots  \tau_N \big[ T_n(\mathbf Y_N) \big] \Big] = O \Big( N ^{  {\mathfrak r(T_1)}/2-1} \dots  N ^{  {\mathfrak r(T_n)}/2-1} \Big).
	\qea
	
\noindent Equivalently, by Proposition/Definition \ref{prop:InjectiveTrace}, one can replace the trace by the injective one.
\end{Ass}

\subsubsection{Proof of Proposition 2.2}

\noindent We go back to the proof of Proposition \ref{Prop:Example}, which tells situations where assumptions are satisfied.

\begin{proof}[Proof of 1.] This claim is a consequence of Mingo and Speicher Theorem \ref{Th:MS}.
\end{proof}

\begin{proof}[Proof of 2.] We just remark that for any $^*$-test graph $T$ there exists a polynomial $P$ such that $\tau_N\big[T(\mathbf D_N) \big] = \frac 1 N \Tr \big[P(\mathbf D_N) \big]$.
\end{proof}

\begin{proof}[Proof of 3.]
\noindent By \cite{MAL12}, the convergence in distribution of traffics of $\mathbf Y_N$ and the weak local convergence of a $\mathbf G_N$ are equivalent. Moreover, for any $^*$-test graph $T$, there exists a cyclic $^*$-test graph $\tilde T$ such that $\tau_N\big[T(\mathbf G_N)\big] = \tau_N\big[\tilde T(\mathbf G_N) \big]$. Hence, Assumption 2 implies Assumption 3.
\end{proof}

\subsection{The traffic-asymptotic freeness of large random matrices}

\subsubsection{Definition}\label{sec:DefFree}
	
\noindent Recall the definition from \cite{MAL12} (in an slight different formulation).

\begin{Def}[Traffic-asymptotic freeness]~\label{def:FreeProdGraphs}
\begin{enumerate}
	\item {\bf Free product of $^*$-test graphs:} Let $\mathbf x_1 \etc \mathbf x_p$ be families of variables. A $^*$-test graphs $T$ with labels in $\mathbf x = (\mathbf x_1 \etc \mathbf x_p)$ is said to be a free product in $\mathbf x_1 \etc \mathbf x_p$ whenever it has the following structure (see Figure \ref{fig:DefFalseFreeness}). Denote by $T_1 \etc T_K$ the connected components of $T$ that are labelled with variables in a same family. Consider the undirected graph $G_{red}(T)$ defined by:
\begin{itemize}
	\item the vertices of $G_{red}(T)$ are $T_1 \etc T_K$ with in addition the vertices $v_1 \etc v_L$ of $T$ that are common to many components $T_1 \etc T_K$,
	\item there is an edge between $T_i$ and $v_j$ if $v_j$ is a vertex of $T_i$, $i=1\etc K$, $j=1\etc L$.
\end{itemize}

\noindent Then, $T$ is a free product in $\mathbf x_1 \etc \mathbf x_p$ whenever $G_{red}(T)$ is a tree.
	\item {\bf Traffic-asymptotic freeness:} Let $\mathbf X_1 \etc \mathbf X_p$ be families of $N$ by $N$ random matrices, whose entries have all their moments, and having jointly a mean limiting distribution of traffics on $\Gcyc$, that is: for any cyclic $^*$-test graphs $T$ in variables $\mathbf x_1 \etc \mathbf x_p$, 
	\eq
		\tau[T] := \Nlim \esp \Big[ \tau_NÊ\big[ T(\mathbf X_1 \etc \mathbf X_p) \big] \Big] \textrm{ exists.}
	\qe
	
\noindent We say that $\mathbf X_1 \etc \mathbf X_p$ are asymptotically traffic-free whenever: for any cyclic $^*$-test graphs $T$ in variables $\mathbf x_1 \etc \mathbf x_p$:
\begin{itemize}
	\item if $T$ is a free product in $\mathbf x_1 \etc \mathbf x_p$, then
		\eq
			\tau^0[T] =  \prod_{\tilde T}\tau^0[\tilde T],
		\qe
	where the product is over the connected components of $T$ that are labelled with variables in a same family.
	\item otherwise, $\tau^0[T] =0.$
\end{itemize}
	
\end{enumerate}
\end{Def}
	
	\begin{figure}[!h]
		\parbox{10.5cm}{ \includegraphics[width=100mm]{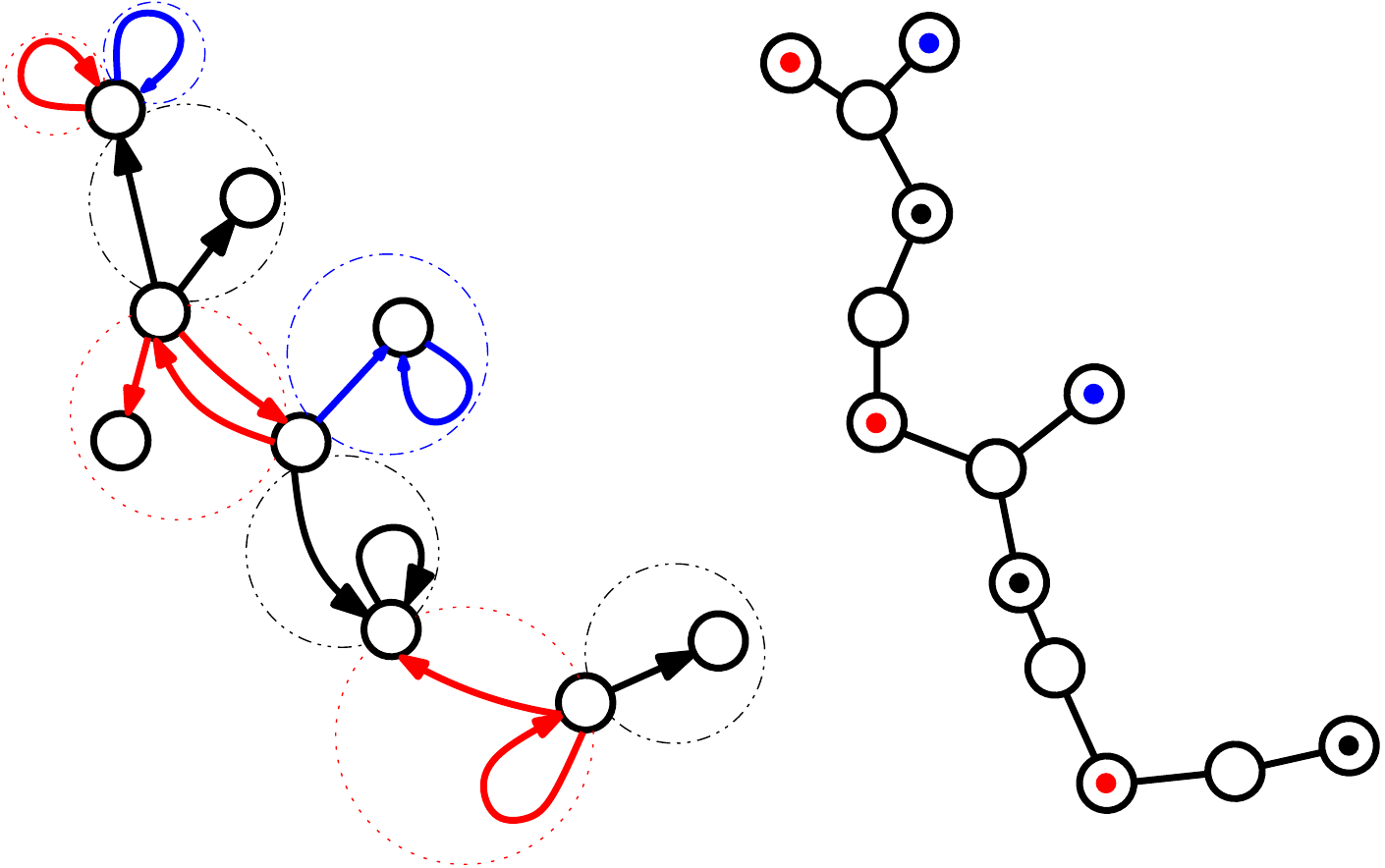}}
		\parbox{4cm}{ \caption{Left: a $^*$-test graph $T$ free product in three families, represented by different colors. The non trivial connected components labelled by a same family of variables are encircled. Right: the graph $G_{red}(T)$. For convenience, the vertices corresponding to components of $T$ are marked with a dot of the corresponding color.}
		\label{fig:DefFalseFreeness}
		}
	\end{figure}

\subsubsection{A traffic-asymptotic freeness theorem on $\Gcyc$}\label{sec:AsymFreeTh}

\begin{Th}[The asymptotic freeness of $\mathbf X_1\toN \etc \mathbf X_p\toN$ on $\Gcyc$]~\label{Th:AsympFree}
\\Let $\mathbf {X}_1\toN \etc \mathbf {X}_p\toN$ be families of $N$ by $N$ random matrices. Assume the following.
\begin{enumerate}
	\item {\bf Statistical independence:}
	\\The families $\mathbf {X}_1\toN \etc \mathbf {X}_p\toN$ are statistically independent.
	
	\item {\bf Joint invariance by permutation:}
	\\For any permutation matrix $U_N$, and any $j=1\etc p$ except possibly one,
			\begin{equation}
				U_N \mathbf {X}_j\toN U_N^*  \overset{\mathcal L}= \mathbf {X}_j\toN.
			\end{equation}

	\item {\bf Convergence in mean distribution of traffics on $\Gcyc$:}
	\\For any $j=1\etc p$, the family $\mathbf X_j^{(N)}$ satisfies Assumption \ref{ConvTraff}.
		
	\item {\bf Technical condition:}
	\\For any $j=1\etc p$, the family $\mathbf X_j^{(N)}$ satisfies Assumptions 2 and \ref{TechAss}.
	
\end{enumerate}

\noindent Then, the joint family $ (\mathbf X_1\toN \etc \mathbf X_p\toN)$ has a mean limiting distribution of traffics on $\Gcyc$. The families of matrices $\mathbf X_1\toN \etc \mathbf X_p\toN$ are asymptotically traffic-free.
\end{Th}

\noindent This theorem is a slight different version of \cite[Theorem 12.1]{MAL12}, where we do not assume Assumption 3 but the convergence of the distribution of traffics for all $^*$-test graphs.

\begin{proof} Let $T=(V,E,\gamma,\eps)$ be a cyclic $^*$-test graph in variables $\mathbf x_1 \etc \mathbf x_p$. For $i=1\etc p$, we denote by $T_{i,k}$, $k=1\etc K_i$ the connected components of $T$ labelled by variables in $\mathbf x_i$ that are not reduced to a single vertex. In general, these $^*$-test graphs are not cyclic. For $i=1 \etc p$, denote by $V_i$ the union of the vertices of $T_{i,k}$ for $k=1\etc K_i$. Then, by $p-1$ applications of \cite[Proposition 12.3]{MAL12}, one has

\begin{eqnarray} \label{FirstEq}
	  \esp\Big[ \tau_N^0\big[ T(\mathbf X_N)\big] \Big]  & = & \frac{(N-1)!}{(N-|V|)!} \times   \frac{(N-|V_1| )! \dots (N-|V_p| )!} { (N-1)! \dots (N-1)! } N^{K_1-1} \dots N^{K_p-1} \nonumber \\
	&	  \times &  \esp \Big[ \prod_{k=1}^{K_1}   \tau_N^0 \big[ T_{1,k} (\mathbf X_1\toN) \big] \Big]   \times \dots \times  \esp \Big[  \prod_{k=1}^{K_p}   \tau_N^0 \big[ T_{p,k} (\mathbf X_1\toN) \big] \Big]. 
\end{eqnarray}

\noindent If $T$ is a free product of $^*$-test graphs in variables $\mathbf x_1 \etc \mathbf x_2$, then $T_{i,k}$ are cyclic for any $i=1\etc p$ and $k=1\etc K_i$. By the convergence in distribution of traffics on $\Gcyc$ of $\mathbf X_1 \etc \mathbf X_p$ separately (and Assumption \ref{ConvTraff}), we get the convergence of $ \esp\big[ \tau_N^0\big[ T(\mathbf X_N)\big] \big]$ to the expected limit with minor modification of the proof of \cite[Theorem 12.1]{MAL12}.
\\
\\From now, we assume that $T$ is not a free product and show the following (it will be useful later), with the same notations as above.

\begin{Lem}[Tightness on the setting of the asymptotic traffic-freeness Theorem on $\Gcyc$]
\noindent With $\mathbf X_N$ as in Theorem \ref{Th:AsympFree} except that Assumption \ref{ConvTraff} is not satisfied, for any cyclic $^*$-test graph $T$ which is not a free product, the quantity $\esp\big[ \tau_N^0\big[ T(\mathbf X_N)\big]  \big|$ is $O(N^{-1})$ as $N$ goes to infinity. 
\end{Lem}

\noindent First, remark that
	\eq
		 \frac{(N-1)!}{(N-|V|)!} \times   \frac{(N-|V_1| )! \dots (N-|V_p| )!} { (N-1)! \dots (N-1)! } N^{K_1-1} \dots N^{K_p-1}  \sim N^{K_1 + \dots K_p +|V| -\big( |V_1| + \dots + |V_p| \big) -1} =: N^{-\rho}.
	\qe

\noindent Write $|V| = v_1 + v_2$, where $v_1$ is the number of vertices of $T$ that belong to a single connected component $T_{i,k}$, $i=1\etc p$ and $k=1\etc K_i$. Similarly, for any $i=1\etc p$ and $k=1\etc K_i$, write $|V_i| = v^{(i)}_1 + v^{(i)}_2$, where $v^{(i)}_1$ is the number of vertices of $T_{i,k}$ that do not belong to other connected components. The number of vertices of $G_{red}(T)$ is $v_{red} = K_1 + \dots + K_p + v_2$, and its number of edges is $e_{red} = \sum_{i,k} v^{(i)}_2$. Hence, 
	\eq
		 - \rho & := & K_1 + \dots K_p +|V| -\big( |V_1| + \dots + |V_p| \big) -1\\
		 	& = & K_1 + \dots K_p + v_1 + v_2 -\sum_{i,k} (  v^{(i)}_1 + v^{(i)}_2 ) -1\\
			& = &  K_1 + \dots K_p +v_2 - \sum_{i,k} v^{(i)}_2 -1 = v_{red} - e_{red} -1.
	\qe

\noindent By the relation between the number of vertices and edges in a connected graph \cite[Lemma 1.1]{GUI}, $\rho$ is the number of cycles of $G_{red}(T)$, that is the maximal number of edges than can be removed from $G_{red}(T)$ while keeping a connected graph.
\\
\\On the other hand, by Lemma \ref{prop:InjectiveTrace}, for any $i=1\etc p$ and $k=1\etc K_i$, one has the relation
	\eq
		 \tau_N^0 \big[ T_{i,k} (\mathbf X_i\toN) \big] = \sum_{\pi \in \mathcal P(V_{i,k})}  \tau_N \big[ T_{i,k}^\pi (\mathbf X_i\toN) \big]  \mu_{V_{i,k}}(\pi),
	\qe
	
\noindent where $V_{i,k}$ stands for the set of vertices of $T_{i,k}$. Since for any $\pi$ in $\mathcal P(V_{i,k})$ one has $\mathfrak r( T_{i,k}^\pi) \leq \mathfrak r( T_{i,k})$, by Assumption \ref{TechAss} there exists a constant $C$ such that 
		\eq
			\bigg| \esp \Big[  \tau_N^0 \big[ T_{i,k} (\mathbf X_i\toN) \big] \Big] \bigg|  \leq C N ^{  {\mathfrak r(T_{i,k})}/2-1}.
		\qe
		
\noindent So, by the formula \eqref{FirstEq} for $\esp\big[ \tau_N^0\big[ T(\mathbf X_N)\big] \big] $ and the equivalent of the normalizing factor, we get that 
	\eq
		\esp\big[ \tau_N^0\big[ T(\mathbf X_N)\big] \big] &  = & O \big( N^{-\rho} \big) \times O \big( N^{  \sum_{i,k}({\mathfrak r(T_{i,k})}/2-1)} \big) =:  O\big( N^{-\delta} \big).
	\qe

\noindent Since $T$ is not a free product, there exists a cycle on $G_{red}(T)$. Moreover, $T$ being cyclic, each $^*$-test graph $T_{i,k}$ whose tree of two-edge connected components has $\ell$ leaves is responsible of the addition of $(\ell -2)/2$ cycles in $G_{red}(T)$, so that the total number $\rho$ of cycles in $G_{red}(T)$ satisfies
	\eq
		\rho \geq 1 + \sum_{i,k} \frac{ \mathfrak r(T_{i,k}) -2}2.
	\qe

\noindent Hence we get that $\delta \geq 1$, so $\esp\big[ \tau_N^0\big[ T(\mathbf X_N)\big] \big] = O (N^{-1}) $ as expected.

\end{proof}

\subsubsection{The limiting distribution of traffics of a single heavy Wigner matrix} \label{sec:DistributionHeavyWigner}

\noindent In this section we prove Proposition \ref{Prop:DistrHeavy} and show that a single random matrix satisfies Assumptions 2 and 3.
\\
\\As we consider a single Hermitian matrix, it is sufficient to consider $^*$-test graph of the form $T=(V,E)$ (the maps $\gamma$ and $\eps$ are trivial). By invariance of $X_N$ by conjugacy by a permutation matrix,
\eq
	\esp \Big[ \tau^0_N\big[ T(X_N) \big] \Big] & = & \frac 1 N \sum_{ \substack{ \phi : V \to \{1 \etc N\} \\ \textrm{injective} } } \esp \bigg[ \prod_{e\in E} X_N\big( \phi(e) \big)\bigg] \\
			& = & \frac{ (N-1)!}{(N-|V|)!} \times \delta^0_N \big[ T(X_N) \big], 
\qe
where $\delta^0_N \big[ T(X_N) \big] = \esp \big[ \prod_{e\in E} X_N\big( \phi(e) \big)\big] $ does not depend on the injective map $\phi$.  For any $k\geq 1$, denote by $p_{k}$ the number of vertices of $T$ where are attached $k$ loops. For any $k_1 \geq k_2\geq 1$, denote by $q_{k_1,k_2}$ the number of pairs of vertices with $k_1$ edges attaching these vertices in one way and $k_2$ others in the opposite direction. Then, by independence of the entries of $X_N$,
\eq
		\delta^0_N \big[ T(X_N) \big] = \prod_{k\geq 1} \bigg(  \frac{Ê\int t^{k} \textrm d\nu_N(t)}{ N^{\frac k2}} \bigg)^{p_k} \prod_{k_1\geq k_2\geq 1} \bigg(  \frac{  \int z^{k_1} \bar z^{k_2} \textrm d\mu_N(z)} { N^{\frac k2}} \bigg)^{q_{k_1,k_2}} 
\qe
Denote
		$$B =  \sum_{k\geq 1} p_k +  \sum_{k_1,k_2\geq 1}q_{k_1,k_2},$$
which is the number of egdes of $T$ when the multiplicity and the orientation are forgotten. Then, one has
\eq
		\frac 1 {N^B} \delta^0_N \big[ T(X_N) \big] = \prod_{k\geq 1} \bigg(  \frac{Ê\int t^{k} \textrm d\nu_N(t)}{ N^{\frac k2-1}} \bigg)^{p_k} \bigg(  \frac{  \int t^{k} \textrm d\mu_N(t)} { N^{\frac k2-1}} \bigg)^{q_k}. 
\qe
Since the entries of $X_N$ are centered, $ \delta^0_N \big[ T(M_N) \big]$ vanishes as soon as an edge of $T$ is of multiplicity one, orientation forgotten. By the Cauchy-Schwarz inequality and by assumptions (\ref{Assum:MomentMu}), (\ref{Assum:MomentNu}), for any $k\geq 1$,
		\eq
			\frac{ \int t^{k} \textrm d\mu_N(t)}{N^{\frac k2-1}} & = & O(1)\\
			\frac{ \int t^{k} \textrm d\nu_N(t)}{N^{\frac k2-1}} & = & O(1).
		\qe
Hence, we get that $\frac 1 {N^B} \delta^0_N \big[ T(X_N) \big]$ is bounded. Moreover, if $T$ has no loops and all its edges are of even multiplicity, then 
	\eqa
		\frac 1 {N^B} \delta^0_N \big[ T(X_N) \big] \limN \prod_{k\geq 1}   a_k  ^{q_{2k}}. \label{eq:ProofDistributionHeavy}
	\qea
Recall that $B$ is the number of edges of $T$ when multiplicity and orientation of edges are forgotten. By the relation between number of edges and vertices in a connected graph \cite[Lemma 1.1]{GUI},
		$$|V| \leq B +1$$
with equality if and only if the graph obtained from $T$ when we forget the multiplicity and the orientation of its edge is a tree. In that case, we say that $T$ is a fat tree. We deduce from the identity 
	\eq
		\esp \Big[ \tau^0_N\big[ T(X_N) \big] \Big] & = & \big( N^{|V| -(B+1)} + o(1) \big) \times \frac 1 {N^B} \delta^0_N \big[ T(X_N) \big]
	\qe

\noindent that $\esp\big[ \tau_N^0\big[T(X_N) \big] \big]$ is always bounded. If $T$ is cyclic, since cyclic fat trees have even multiplicity of edges, we get by (\ref{eq:ProofDistributionHeavy})
	\eq
		\esp \Big[ \tau^0_N\big[ T(X_N) \big] \Big] & = & \prod_{k\geq 1}   a_k  ^{q_{2k}} \mathbf 1_{T \textrm{ is a fat tree} } +o(1) \limN \tau^0[T].
	\qe

\noindent It remains that $X_N$ satisfies Assumptions 2 and 3. Let $T_1\etc T_k$ be $^*$-test graphs. Let $T=(V,E)$ be the $^*$-graph ($^*$-test graph without the connectedness condition, for which trace and injective trace are defined by the same formulas) obtained as the disjoint union of $T_1 \etc T_K$. By \cite[Lemma 11.7]{MAL12},
\eq
	 \tau_N^0 \big[ T_1(X_N) \big] \dots  \tau_N^0 \big[ T_n(X_N) \big]  
	  & = &  \sum_{\pi}  \frac 1 {N^{n-1}}   \tau_N^0 \big[ T^\pi(X_N) \big] ,
\qe
\noindent where the sum is over all partitions $\pi$ on $V$ that contain at most one vertex of each $T_k$, $k=1\etc n$. For any such a partition $\pi$, denote by $T_1^\pi \etc \tilde T_{m_\pi}^\pi$ the connected components of $T^\pi$. By the independence of the entries of $X_N$, 
\eq
	  \lefteqn{ \esp \Big[ \tau_N^0 \big[ T_1(X_N) \big] \dots  \tau_N^0 \big[ T_n(X_N) \big]  \Big]}\\
	  &  =  & \sum_{\pi}  \frac {N^{m_\pi}} {N^{n}}  \esp \Big[ \tau_N^0 \big[ T_1^\pi (X_N) \big]  \Big] \dots  \esp \Big[ \tau_N^0 \big[ T_{m_\pi}^\pi (X_N) \big]  \Big] ,
\qe

\noindent Each expectation is bounded and converges as $N$ goes to infinity if the $^*$-test graphs are cyclic. We always has $m_\pi \leq n$, expect for the trivial partition. Hence, $ \esp \big[ \tau_N^0 \big[ T_1(X_N) \big] \dots  \tau_N^0 \big[ T_n(X_N) \big]  \big]$ is bounded, and if the $^*$-graphs are cyclic we get
\eq
	  \esp \Big[ \tau_N^0 \big[ T_1(X_N) \big] \dots  \tau_N^0 \big[ T_n(X_N) \big]  \Big] \limN \tau^0  [ T_1  ] \dots  \tau^0  [ T_n ].
\qe

\subsection{Some consequences}

\begin{Cor}[The Wigner case]\label{Cor:WignerCase}~
\\Consider a family $\mathbf X_N$ of independent Wigner matrices, independent of $\mathbf Y_N$. Assume that $\mathbf Y_N$ converges in $^*$-distribution and that it satisfies Assumptions 2 and 3. Then, $(\mathbf X_N, \mathbf Y_N)$ converges in $^*$-distribution and are asymptotically $^*$-free.
\end{Cor}

\begin{proof} By Assumptions 2 and 3, $\mathbf Y_N$ is tight on $\Gcyc$. Consider a subsequence along which $\mathbf Y_N$ converges. Let $X_N$ be a heavy Wigner matrix with trivial parameter $(a,0,0, \dots)$. By Theorem \ref{MainTh}, $(\mathbf X_N, \mathbf Y_N)$ has a limiting $^*$-distribution along this subsequence, given by formula \eqref{Eq:Distrib1}. The family $\mathbf X_N$ converges to a family of semicircular traffics, traffic free from the limit of $\mathbf Y_N$. Since the traffic-freeness of semicircular variables with arbitrary traffics implies their $^*$-freeness \cite{MAL}, we get that $\mathbf X_N$ and $\mathbf Y_N$ are asymptotically $^*$-free variables.
\end{proof}

\begin{Prop}[Heavy covariance matrices]\label{Prop:HeavyCov}~
\\Let $N$ be an integer. Let $N_0\toN \etc N_K\toN$ be integers such that $N_k\toN \sim c_k N$, $c_k>0$ for any $k=0\etc K$. Let $\mathbf W_N = (W_1\toN \etc W_p\toN)$ be a family of random matrices such that: 
	\begin{itemize}
		\item for any $j=1\etc p$, one has $W_j\toN = M_j\toN Z_j\toN M_j\toNs$. 
		\item $\mathbf M_N$ is a family of independent random matrices $M_j\toN$, $j=1\etc p$, with independent entries having the same distribution of a random variable $m_{j,N}$ such that 
			$$\esp\big[ N m_{j,N}^{2n}\big] \limN a_{j,n}$$
 for any $n\geq 1$. The matrix $M_j\toN$ is of size $N_0$ by $N_{k_j}\toN$ for an integer $k_j$ in $\{1\etc K\}$, $j=1\etc p$.
		\item $\mathbf Z_j\toN = (Z_1\toN \etc Z_p\toN)$ is a family of random matrices, and the matrix $Z_j\toN$ is of size $N_{k_j}\toN$ by $N_{k_j}\toN$.
	\end{itemize}
	
	\noindent Let $\mathbf Y_N$ be a family of $N_0$ by $N_0$ random matrices. Assume that
	\begin{enumerate}
		\item the families of matrices $M_N, \mathbf Y_N, (Z_j\toN )_{k_j=k}, k=1\etc K$ are independent,
		\item the families of matrices $\mathbf Y_N, (Z_j\toN )_{k_j=k}, k=1\etc K$, satisfies the assumption of Theorem \ref{MainTh} separately.
	\end{enumerate}
	
	\noindent Then, the family of matrices $(\mathbf W_N, \mathbf Y_N)$ has a limiting $^*$-distribution as $N$ goes to infinity.
\end{Prop}

\begin{proof} We prove the Proposition for $K=1$, the result can be obtained by recurrence on the number of size of matrices. Consider the square matrices of size $(N_0 + N_1)$, by blocks
	\eq
		&\tilde{ W}_j\toN = \left( \begin{array}{cc}
			  W_j\toN&\\
			&0
		\end{array}\right),
		 \  
		 \tilde{ Y}_j\toN = \left( \begin{array}{cc}
			 Y_j\toN & \\
			&0
		\end{array}\right),
		\\
		 &\tilde{ Z}_j\toN = \left( \begin{array}{cc}
			0 & \\
			&  Z_j\toN\\
		\end{array}\right),
		 \  
		 \tilde{ M}_j\toN = \left( \begin{array}{cc}
			 0 &M_j\toN  \\
			 0 & 0
		\end{array}\right), \ \ j=1\etc p.
	\qe

\noindent We consider the matrices
	\eq
		 &
		e_0 =   \left( \begin{array}{cc}
			 \mathbf 1_{N_0}\toN & \\
			 & 0\\
		\end{array}\right),
		\ \
		&e_{1} =  \left( \begin{array}{cc}
			0 &\\
			& \mathbf 1_{N_1}
		\end{array}\right),
		\ \
		\tilde X_j\toN = \left( \begin{array}{cc}
			X_j^{(0,N)} & M_j\toN \\
			 M_j\toNs & X_j^{(1,N)} 
		\end{array}\right), j=1\etc p,
	\qe

	\noindent where $X_j^{(0,N)}, X_j^{(1,N)}$, $j=1\etc p$, are independent heavy Wigner matrix with parameter \break $\big( \Nlim\esp[ N m_{j,N}^{2k} ] \big)_{k\geq 1}$, square of size $N_0$ and $N_1$ respectively, independent of $(\mathbf Y_N, \mathbf Z_N, \mathbf M_N)$. It can be noted that 
	\eq
		\tau_N\big[ T( \tilde{ \mathbf Y}_N, \tilde{ \mathbf Z}_N, e_1,e_2) \big] & = & \tau_{N_0}\big[ T( \tilde{ \mathbf Y}_N, e_2)\big] \times \frac 1 {\frac{c_0}{c_1}+1} \mathbf 1_{T \textrm{ labelled in } (\mathbf y, e_2)}\\
		&  & \ \ \ \ + \ \tau_{N_0}\big[ T( \tilde{ \mathbf Z}_N, e_1)\big] \times \frac 1 {\frac{c_1}{c_0}+1} \mathbf 1_{T \textrm{ labelled in } (\mathbf z, e_1)}.
	\qe
	
	\noindent Hence, $( \tilde{ \mathbf Y}_N, \tilde{ \mathbf Z}_N, e_1,e_2)$ satisfies the assumption of Theorem \ref{MainTh}, and so $( \tilde {\mathbf X}_N, \tilde{ \mathbf Y}_N, \tilde{ \mathbf Z}_N, e_1,e_2)$ has a limiting $^*$-distribution, given by the traffic free product. Since for $j=1\etc p$, one has $\tilde M_j\toN = e_0 \tilde X_j\toN e_1$ and $\tilde W_j\toN = \tilde M_j\toN \tilde Z_j\toN \tilde M_j\toNs$, we get that $(\tilde {\mathbf W}_N, \tilde {\mathbf Y}_N)$ has a limiting $^*$-distribution. Moreover, for any $^*$-polynomial $P$
	\eq
		\frac 1 {N_1+N_0} \Tr \big[ P (\tilde {\mathbf W}_N, \tilde {\mathbf Y}_N) \big]  = \frac 1{\frac {c_1}{c_0}+1} \frac 1 {N_0} \Tr \big[ P ( {\mathbf W}_N, {\mathbf Y}_N) \big].
	\qe
	Hence the convergence of $( {\mathbf W}_N,  {\mathbf Y}_N)$.
\end{proof}

\begin{Prop}[Independent L\'evy and random matrices]\label{Prop:LevyMatrices}~
\\Let $\mathbf X_N =(X_1\toN \etc X_p\toN)$ be a family of independent L\'evy matrices, independent of a family of random matrices $\mathbf Y_N$ satisfying the assumption of Theorem \ref{MainTh}. Then, for any Hermitian matrix $H_N = P(\mathbf X_N, \mathbf Y_N)$, where $P$ is a fixed $^*$-polynomial, the empirical eigenvalues distribution of $H_N$ converges weakly, i.e. for any continue bounded function $f:\mathbf R \to \mathbf R$, $\esp \Big[ \frac 1 N \Tr\big[ f(H_N) \big] \Big]$ converges.
\end{Prop}

\begin{proof} Let $B\geq 1$ be a positive and large number, and set for any $j=1\etc p$, 
	$$X_j^{(B,N)} = \big(  X_j\toN(m,n) \times \mathbf 1_{ |X_j\toN(m,n) |\leq B}\big)_{m,n=1\etc N}.$$

\noindent By Section \ref{sec:DefLevy}, the family $\mathbf X_N^{(B)} = (X_1^{(B,N)} \etc X_p^{(B,N)})$ is a family of independent heavy Wigner matrices. Hence, by Theorem \ref{MainTh} $(\mathbf X_N^{(B)}, \mathbf Y_N)$ has a limiting $^*$-distribution. Hence, the empirical eigenvalues distribution of $H_N^{(B)} = P( \mathbf X_N^{(B)}, \mathbf Y_N)$ converges weakly to a measure $\mu^{(B)}$. By the same reasoning as in \cite[Section 8]{BGCD12}, $H_N^{(B)}$ and $H_N$ are closed is the sense of rank and by \cite[Lemma 12.2]{BGCD12}, this implies that the empirical eigenvalues distribution of $H_N$ converges weakly to a measure $\mu$ and $\mu = \lim_{B \rightarrow \infty} \mu^{(B)}$.
\end{proof}

\section[Limiting $^*$-moments of heavy Wigner and random matrices]{Limiting $^*$-moments of independent heavy Wigner and random matrices}

\subsection{Proof of Proposition 2.8}

\noindent Let $\mathbf X_N = (X_1\toN \etc X_p\toN)$ and $\mathbf Y_N$ be as in Theorem \ref{MainTh}. We denote by $\Phi$ their mean limiting $^*$-distribution and by $\tau$ their limiting distribution of traffics on $\Gcyc$: for any $^*$-polynomial $P$ and any $^*$-test graph,
	\eqa
		\Phi(P) & = & \Nlim \frac 1 N \Tr \big[ P(\mathbf X_N, \mathbf Y_N) \big],\\
		\tau[P] & = & \Nlim \tau_N \big[ T (\mathbf X_N, \mathbf Y_N) \big].
	\qea

\noindent The family of variables $\mathbf x = (x_1 \etc x_p)$ corresponds to $\mathbf X_N$, the family $\mathbf y$ correspond to $\mathbf Y_N$. Theorem \ref{MainTh} tells us how $\Phi$ can be written in term of the injective version of $\tau$, defined by (\ref{eq:DefInj}). Consider a polynomial $P$ of the form
	\eqa \label{eq:ReferencePoly}
		P = x_{\gamma(1)} P_1(\mathbf y) \dots x_{\gamma(L)} P_L(\mathbf y),
	\qea
where $\gamma: \{1\etc L\} \to \{1\etc p\}$. Let $T_P$ be the $^*$-test graph in variables $\mathbf x, P_1(\mathbf y) \etc  P_L(\mathbf y)$ as in Figure \ref{SimpleCycle2}.

	\begin{figure}[!h]
		\parbox{1.5cm}{~}
		\parbox{4.5cm}{ \includegraphics[height=25mm]{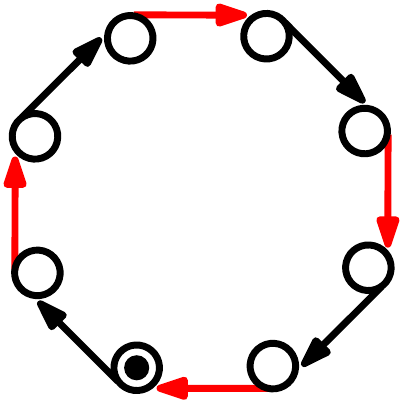}}
		\parbox{8cm}{ \caption{Example with $L=4$. Starting from the vertex marked with a dot and running clock-wisely, the edges are labelled $x_{\gamma(1)}, P_1(\mathbf y)\etc x_{\gamma(L)}$ and $P_L(\mathbf y)$. Then, the black edges have labels $x_{\gamma(1)} \etc x_{\gamma(L)}$, and the red ones have label $P_1(\mathbf y) \etc P_L(\mathbf y)$.} \label{SimpleCycle2}}
	\end{figure}
		
\noindent Then, by the traffic-asymptotic freeness of $X_1\toN \etc X_p\toN$ and $\mathbf Y_N$, we obtain the formula
	\eqa \label{eq:ApplicFree}
		\Phi( P )  =  \sum_{\pi \in \mathcal P(2L)} \mathbf 1_{( T_P^\pi \textrm{ is a free product} )} \prod_{ \tilde T }   \tau^0 [ \tilde T ] ,
	\qea
	
\noindent where 
\begin{itemize}
	\item $\mathcal P(2L)$ denotes the set of partitions of $\{1\etc 2L\}$,
	\item $T_P^\pi$ is the $^*$-test graph defined from $T_P$ and $\pi$ as in Figure \ref{SimpleCycledef.pdf},
	\item the notion of free product of $^*$-test graphs, given in Definition \ref{def:FreeProdGraphs}, is relatively to the family of variable $(x_1) \etc (x_p), \big(P_1(\mathbf y) \etc P_K(\mathbf y) \big)$,
	\item $\tau^0$ is the injective version of $\tau$, defined by (\ref{eq:DefInj}),
	\item the product is over all connected components of $T_P^\pi$ that are labelled by a family among $(x_1) \etc (x_p), \big(P_1(\mathbf y) \etc P_L(\mathbf y)\big)$, as illustrated in Figure \ref{fig:DefFalseFreeness}.
\end{itemize}

\noindent We have drawn two examples of free products of $^*$-test graphs in Figure \ref{fig:FatFreeProduct2}, remembering the marked point of $T_P$ as we did in Figure \ref{SimpleCycle2}.
\\
\\Let $\pi$ be partition in $\mathcal P(2L)$ such that $T_P^\pi$ is a free product. Let $F^\pi$ be the $^*$-test graph obtained from $T_P^\pi$ by merging the components labelled $P_1(\mathbf y) \etc P_L(\mathbf y)$. Hence, $F^\pi$ must be a fat tree. From $\pi$ and $F^\pi$, we get a minimal cycle on a tree $(G_\pi, c_\pi)$ in $\mathcal L^{(\gamma)}$ as we take care of the way we fold $T_P$ into $T_P^\pi$. The two partitions of Figure \ref{fig:FatFreeProduct2} give the same minimal cycle.

\begin{figure}[h!]
\begin{center}
\includegraphics[width=80mm]{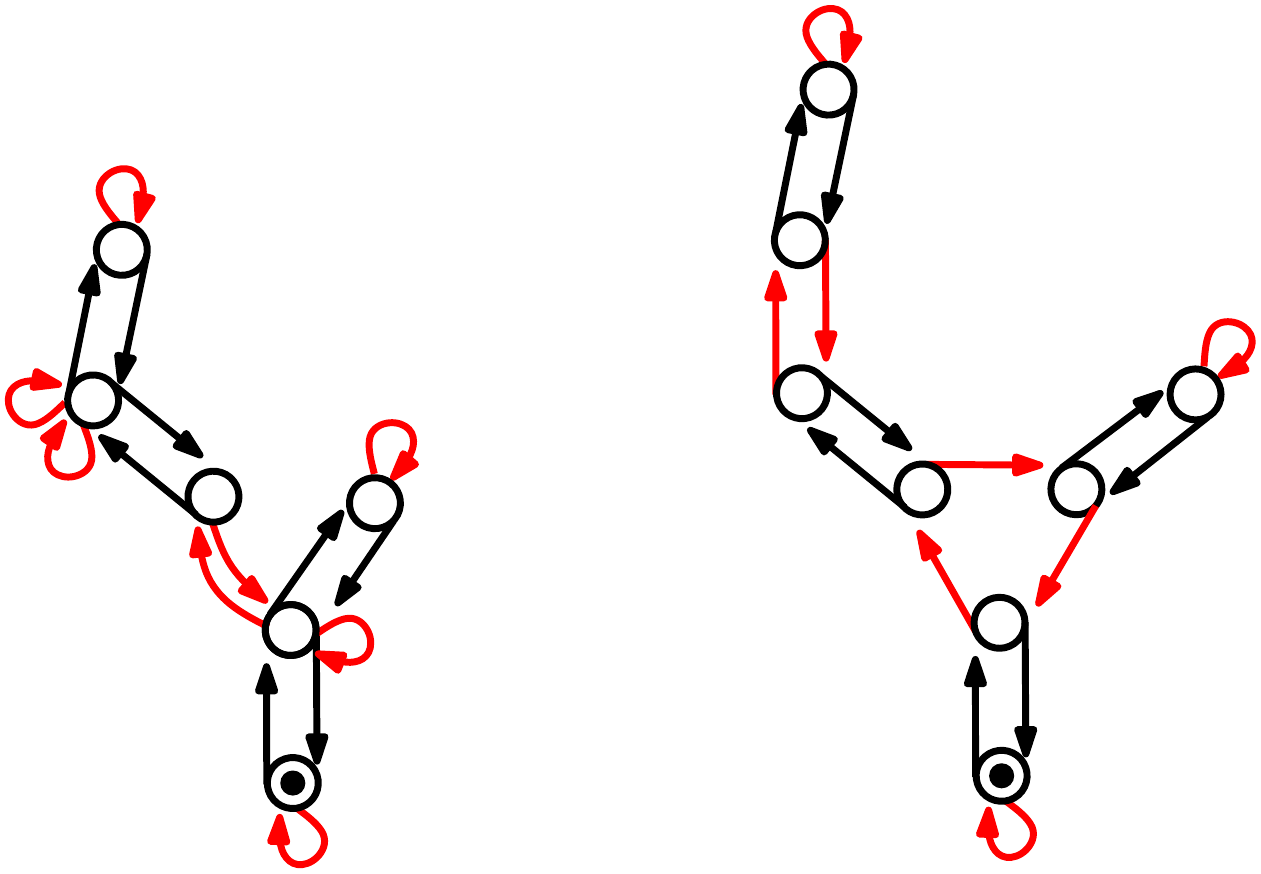} \caption{\label{fig:FatFreeProduct2} Two free products of fat trees (black) and arbitrary $^*$-test graphs (red) for a polynomial $P = x_{\gamma(1)} P_1(\mathbf y) \dots x_{\gamma(L)} P_L(\mathbf y)$ with $L=8$.}
\end{center}
\end{figure}

\noindent Let $(G,c)$ in $\mathcal L^{(\gamma)}$. All partitions $\pi$ such that $(G_\pi,c_\pi) = (G , c)$ will give the same contribution $\omega_{HW}(G, c)$ from heavy Wigner matrices that can be factorized in $\prod_{ \tilde T }   \tau^0 [ \tilde T ]$: 
	\eq
		\Phi( P )  = \sum_{(G,c) \in \mathcal L^{(\gamma)}} \omega_{HW}(G, c)    \sum_{\pi \in \mathcal P(2L)} \mathbf 1_{(G_\pi,c_\pi) = (G , c)}    \prod_{ T' }    \tau^0 [  T' ] ,
	\qe

\noindent where the product on $T'$ is now over all connected components of $T_P^\pi$ that are labelled $P_1(\mathbf y) \etc P_L(\mathbf y)$. Let $(G,c)$ in $\mathcal L^{(\gamma)}$. It remains to show that
	$$\sum_{\pi \in \mathcal P(2L)} \mathbf 1_{(G_\pi,c_\pi) = (G , c)}    \prod_{ T' }    \tau^0 [  T' ]  = \omega_{TR}(G,c).$$
	
\noindent For any $v$ vertex of $G$, recall that we have defined a $^*$-test graphs $T_v$ labelled in $P_1(\mathbf y) \etc P_K(\mathbf y)$. All the partitions $\pi$ such that $(G_\pi,c_\pi) = (G , c)$ give the same fat tree $F = F^\pi$. ''Replace`` the vertices of $F$ by corresponding $^*$-test graphs $T_v$'s in the following way: 
\begin{enumerate}
	\item consider the disjoint union of the $T_v$'s.
	\item By construction, each vertex of a $T_v$ is associated to an edge of $G$. Link the vertices of two different $T_v$ and $T_w$ that correspond to a same edge of $G$ by $n$ edges, where $n$ is the number of times $c$ walks on this edge.
	\item Orient half of these edges in one direction and the other and the other direction.
	\item Label these edges by the color of the corresponding step of $c$.
\end{enumerate}

\noindent The $^*$-test graph we obtain is $T^{\pi_0}_P$, where $\pi_0$ is the coarser partition for which $(G_{\pi_0},c_{\pi_0}) = (G , c)$ (see the rightmost $^*$-test graph in Figure \ref{fig:FatFreeProduct2}). The other partitions $\pi$ which give $(G_\pi,c_\pi) = (G , c)$ are the sub-partitions of $\pi_0$ which do not put in a same block vertices from different $T_v$'s (as for the leftmost $^*$-test graph in Figure \ref{fig:FatFreeProduct2}, compared to the rightmost). By the relation (\ref{eq:InjStand}) between the standard and the injective trace, we obtain the expected result and hence the Proposition.

\subsection{Proof of Corollary 2.9}

\begin{proof}[Proof of 1. and 2.]As a cycle visits a tree with different colors for each vertices, the total number of steps of a given color is an even number. We then obtain that $\Phi ( x_j ) = 0$ for any $j=1\etc p$ and the second point of the corollary. Moreover, $\Phi ( x_j^2 )$ is $a_{j,1}$ since there is only one cycle running on a tree with one edge in two steps, which gives the contribution $a_{j,1}$.
\\
\\{\it Proof of 3.} The set of cycles running on a tree with $n_1$ steps of colors $i_1$, then $n_2$ steps of colors $i_2$, and so on, is in bijection with the product of the sets of cycles running on a tree with $n_j$ steps of colors $i_j$, $j=1\etc L$ as soon as the colors are distinct. The weights $\omega_{HW}$ for $x_{i_1}^{n_1} \dots x_{i_L}^{n_L} $ are the products of weights for $x_{i_1}^{n_1}$ \etc $x_{i_L}^{n_L}$. The weights $\omega_{TR}$ comes from the root of the trees.

\end{proof}

\subsection{The false freeness property}

\noindent The folding trick gives a unformal algorithm for the enumeration of the cycles running on trees when $L$ is not to large, and then for the computation of limiting joint moments in heavy Wigner and deterministic matrices of few degree. Denote by $\mathcal L^{(\ell)}_2$ the elements of $\mathcal L^{(\ell)}$ such that the cycles visits each edge of their tree twice. They correspond to cycles that contribute for classical Wigner matrices, a case where $\mathbf X_N$ and $\mathbf Y_N$ are asymptotically free in the sense of Voiculescu \cite{NS}. The elements of $\mathcal L^{(\ell)} \setminus \mathcal L^{(\ell)}_2$ are enumerated by folding the combinatorial objects of $\mathcal L^{(\ell)}_2$.
\begin{figure}[!h]
\begin{center}
 \includegraphics[width=140mm]{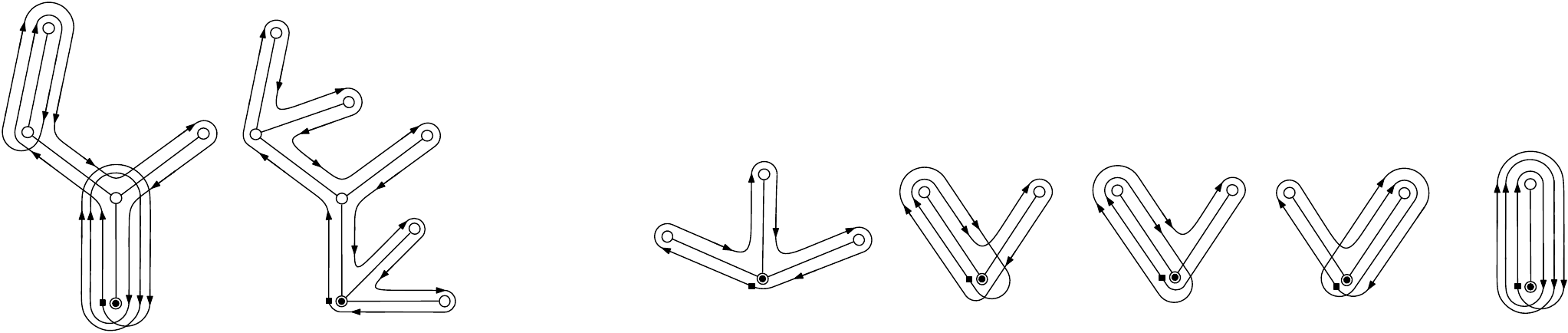}
 \caption{Left: a couple $(G,c)$ and its unfolded version. Right: starting with a double graph, all the couples $(G,c)$ obtained by the folding trick.}\label{fig:RunningOnTheTree3}
 \end{center}
\end{figure}
\\{\bf Unfolding trick: }Let $(G,c)$ in $\mathcal L^{(\ell)}\setminus \mathcal L^{(\ell)}_2$. After some steps, leaving a vertex $v$, the cycle $c$ comes back in an edge it has already visited. Then it induces a sub-cycle $\hat c$ on the tree of the descendent of $v$. We create a copy $\hat G$ of the sub-tree induces by $\hat c$, forget its original embedding and embed it in such a way $\hat c$ respects the rules concerning the order of visits of the edges of $\hat G$. Then we attach $\hat G$ endowed with this new orientation at the vertex $v$, between the edges it has already visited and the others. If some edges of the tree of the descendent of $s$ where only visited by $\hat c$, then we erase them. We then keep an element of $\mathcal L^{(\ell)}$. Iterating this procedure a finite number of times, we then get an element of $\mathcal L^{(\ell)}_2$.
\\
\\{\bf Folding trick: }Reciprocally, let $(G,c)$ be an element of $\mathcal L^{(\ell)}$. Chose an edge $e_1$ of the tree. If possible, chose an other edge $e_2$, which shares the same vertex toward the root and which is of the same color as $e_1$. Then, merge these two edges, draw the tree of the descendant of $e_1$ at the right of the the tree of the descendant of $e_2$ and redirect the cycle $c$ in this new tree. We then obtain an new element of $\mathcal L^{(\ell)}$. For any element $(G_0,c_0)$ of $\mathcal L^{(\ell)}_2$, we denote by $fold(G_0,c_0)$ the set of all elements of $\mathcal L^{(\ell)}$ we get by applying many times this trick.
\\
\\Folding and unfolding tricks are illustrated in Figure \ref{fig:RunningOnTheTree3}. Two different elements of $\mathcal L^{(\ell)}_2$ have different folding sets. We then get from this construction the following proposition.

\begin{Prop}[The false freeness property]\label{prop:FalseFreenessProperty}~\label{Prop:FalseFreeness}
\\For any $^*$-polynomial $P$ of the form
	$P = x_{\gamma(1)} P_1(\mathbf y) \dots x_{\gamma(L)} P_L(\mathbf y)$,
one has
	\eqa		\label{Eq:Distrib1}
		\Phi ( P) = \sum_{ (G_0,c_0) \in \mathcal L_2^{(\gamma)}} \sum_{(G,c) \in fold(G_0,c_0) }   \omega_{HW}(G, c) \times  \omega_{TR}(G, c),
	\qea
where $ \omega_{HW}$ and $ \omega_{TR}$ are as in Proposition \ref{Prop:ComputMoments}.
\end{Prop}

\noindent The false freeness property gives a method to reasonably compute limiting joint moments of heavy Wigner and deterministic matrices:
\begin{enumerate}
	\item Enumerate the elements of $\mathcal L^{(\gamma)}_2$.
	\item Fold the branches of these colored trees.
	\item Then, read the contribution of each element.
\end{enumerate}

\noindent We do not describe how to be sure to obtain all the elements of $\mathcal L^{(\gamma)}$ during the second step of the algorithm, as our purpose is to use this method for relatively small $L$. As an example, we have computed $\Phi( x_1^6 )= 5 a_{1,1}^3 + 6 a_{1,2} a_{1,1}^2 + a_{1,3}  $ in Figure \ref{FatColorCycle}.

\begin{figure}[!h]
\begin{center}
 \includegraphics[width=140mm]{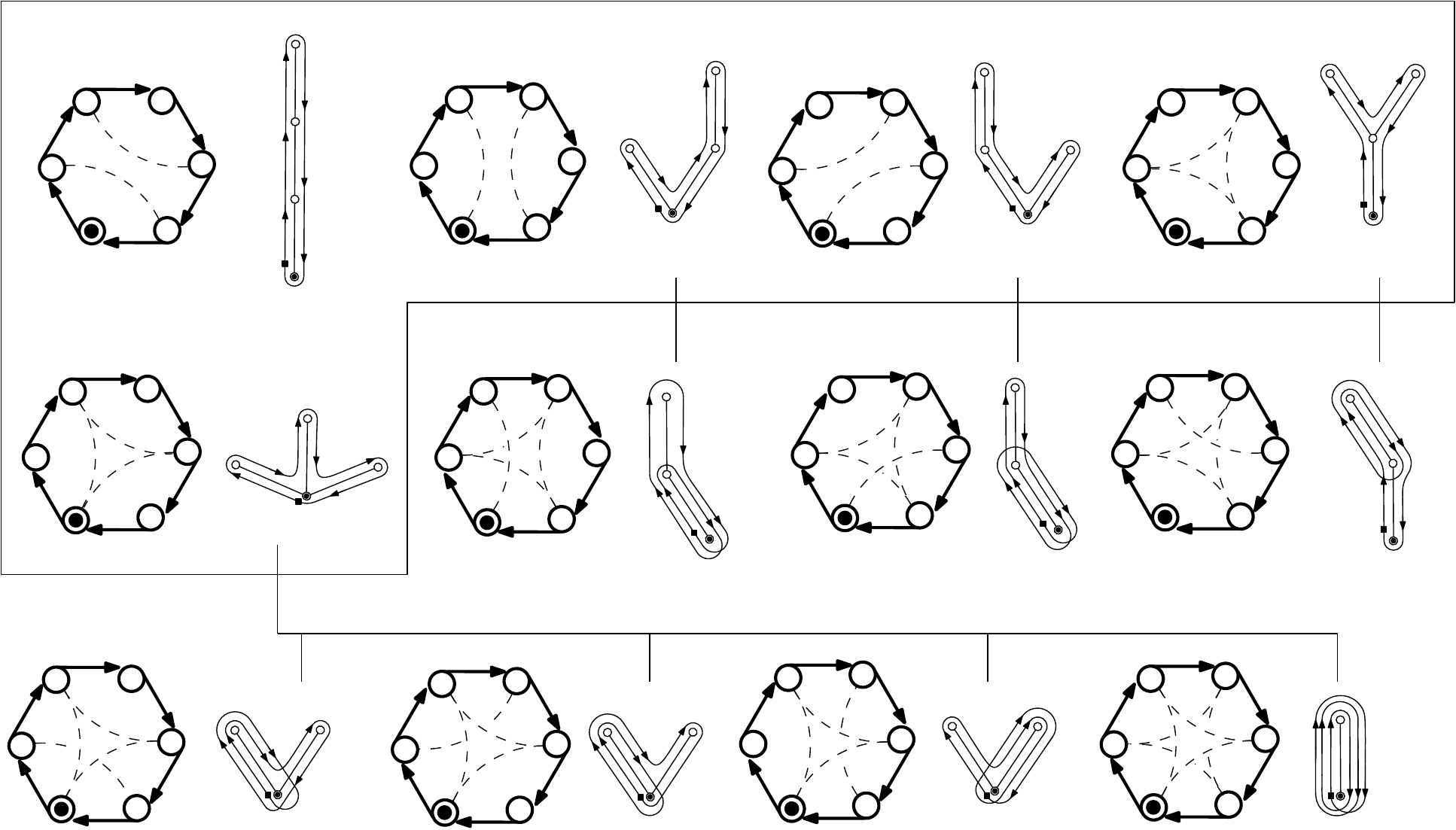}
 \caption{Computation of : $\Phi( x_1^6 )= 5 a_{1,1}^3 + 6 a_{1,2} a_{1,1}^2 + a_{1,3}  $. In the frame, we have enumerate the elements of $\mathcal L^{(1, \dots,1)}_2$. Three of them, the rightmost ones, can be folded in an unique way. The double tree on the bottom of the frame can we folded in four ways, as in Figure \ref{fig:RunningOnTheTree3}. On the left of each minimal cycle on a tree $(G,c)$, we have drawn the permutation $\pi$ of $\{1\etc 6\}$ such that $(G_\pi,c_\pi) =(G,c)$, with the notations of the proof of Proposition \ref{Prop:ComputMoments}. It should be noticed that the partitions on the frame are the dual, in the sense of planar partition, of non crossing pair partitions. Planarity is broken when trees are folded. Then, these partitions are the {\it dual clickable partitions} (see \cite{RYA}).}
 \label{FatColorCycle}
 \end{center}
\end{figure}

\subsection{Proof of Propositions 2.10 and 2.11: the non asymptotic $^*$-freeness of $\mathbf X_N$ and $\mathbf Y_N$}

\begin{proof}[Proof of $f(x_1,x_2) = a_{1,2} a_{2,2}$, where $x_1,x_2$ are heavy Wigner of parameters $(a_{1,k})_{k\geq 1}, (a_{2,k})_{k\geq 1}$] We first expand the quantity
	\eq	
		\lefteqn{ \Phi  \Big (  \big( x_1^2 - \Phi(x_1^2) \big) \big( x_2^2 - \Phi(x_2^2) \big)   \big( x_1^2 - \Phi(x_1^2) \big) \big( x_2^2 - \Phi(x_2^2) \big)  \Big) } \\
	&= &		\Phi( x_1^2 x_2^2 x_1^2 x_2^2) \\
	& & - 2 \Phi(x_1^2) \Phi(x_1^2 x_2^4) - 2 \Phi(x_2^2) \Phi(x_2^2 x_1^4)\\
	& & + 4 \Phi(x_1^2) \Phi(x_2^2) \Phi(x_1^2 x_2^2) + \Phi(x_1^2)^2\Phi(x_2^4) + \Phi(x_2)^2 \Phi(x_1^4)\\
	& & - 4 \Phi(x_1^2)^2 \Phi(x_2^2)^2\\
	& & + \Phi(x_1^2)^2 \Phi(x_2^2)^2.
	\qe
	
\noindent Using the traciality of $\Phi$ (that is $\Phi(PQ) = \Phi(QP)$) and interchanging the roles played by $x_1$ and $x_2$, it is enough to compute $\Phi(x_1^2)$, $\Phi(x_1^2 x_2^2)$, $\Phi(x_1^4)$, $\Phi(x_i^4 x_j^2)$ and $\Phi(x_1^2 x_2^2 x_1^2 x_2^2)$. By the basic properties stated in Corollary \ref{Cor:BasicProp}, we have $\Phi(x_1^2) = a_{1,1}$, $\Phi(x_1^2 x_2^2) = \Phi(x_1^2) \Phi( x_2^2) =a_{1,1} a_{2,1}$ and $\Phi(x_1^4 x_2^2) = \Phi(x_1^4 ) \Phi(x_2^2) = \Phi(x_1^4 ) a_{2,1}$.
\begin{figure}[!h]
\parbox{.5cm}{~}
\parbox{10cm}{  \includegraphics[width=90mm]{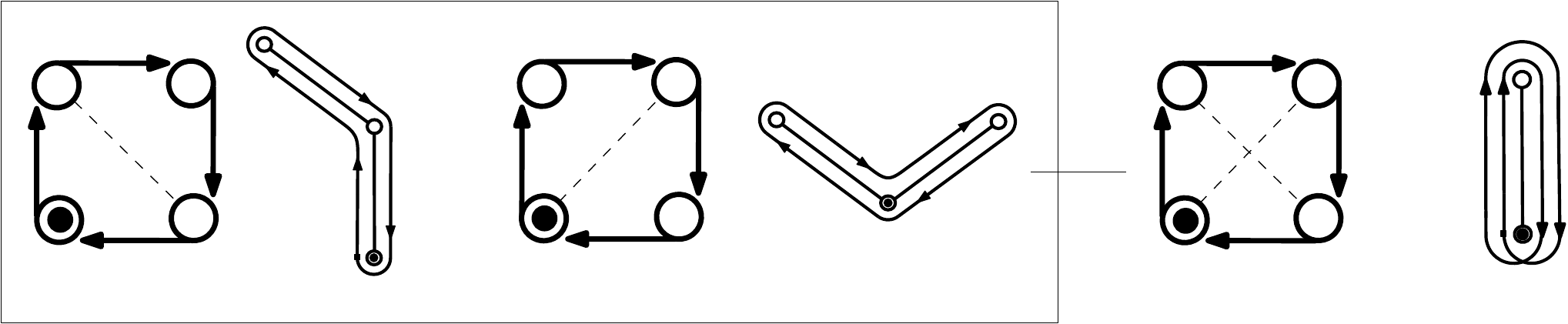}}
\parbox{4cm}{ \caption{\label{CyclesFat} Computation of $\Phi(x_1^4) = 2a_{1,1}^2+ a_{1,2}$.} }
\end{figure}
\begin{figure}[!h]
\begin{center}
 \includegraphics[width=140mm]{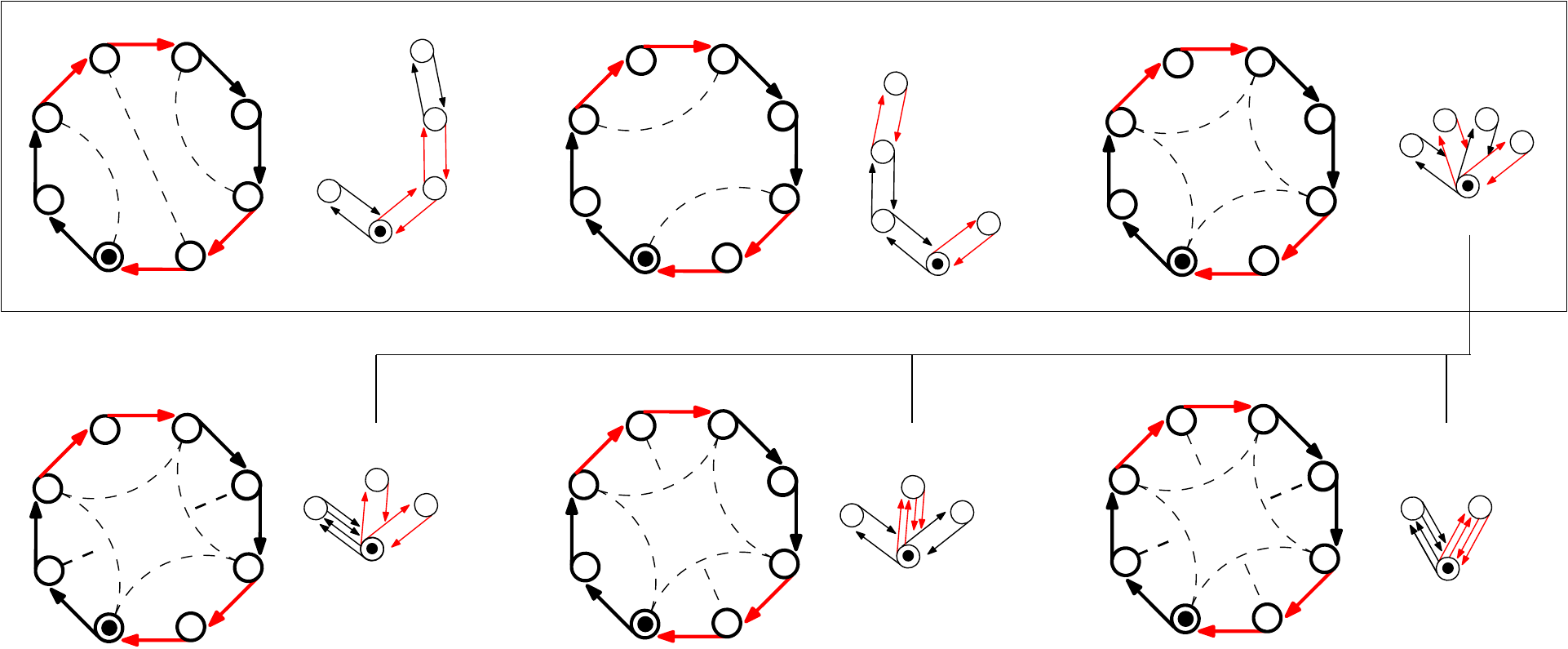}
 \caption{Computation of $\Phi( x_1^2 x_2^2 x_1^2 x_2^2) = 3 a_{1,1}^2 a_{2,1}^2 + a_{1,2} a_{2,1}^2 + a_{1,1}^2 a_{2,2} + a_{1,2} a_{2,2} $. Edges corresponding to $x_1$ are in black, the ones corresponding to $x_2$ are in red. For technical reasons, we have drawn fat trees instead of cycles running on trees.}
 \label{FatColorCycle2}
 \end{center}
\end{figure}
\\
\\The computation of $\Phi(x_1^4 )$ and $\Phi( x_1^2 x_2^2 x_1^2 x_2^2)$ are done in Figures \ref{CyclesFat}  and \ref{FatColorCycle2}, following the algorithm of the false freeness property (Proposition \ref{Prop:FalseFreeness}). As we do not consider deterministic matrices, we do not need to open boxes in the vertices as in Figure \ref{fig:NodesToTestGraphs3}. This gives 
	\eq	
		\lefteqn{ \Phi  \Big (  \big( x_1^2 - \Phi(x_1^2) \big) \big( x_2^2 - \Phi(x_2^2) \big)   \big( x_1^2 - \Phi(x_1^2) \big) \big( x_2^2 - \Phi(x_2^2) \big)  \Big) } \\
	& = & 3 a_{1,1}^2 a_{2,1}^2 + a_{1,2} a_{2,1}^2 + a_{1,1}^2 a_{2,2} + a_{1,2} a_{2,2} \\
	& & - 2a_{1,1}( 2 a_{2,1}^2 + a_{2,2}) - 2a_{2,1}( 2 a_{1,1}^2 + a_{1,2}) \\
	& & + 4 a_{1,1}^2 a_{2,1}^2 + a_{1,1}( 2 a_{2,1}^2 + a_{2,2}) + a_{2,1}( 2 a_{1,1}^2 + a_{1,2}) \\
	& & - 3 a_{1,1}^2 a_{2,1}^2\\
	& & =a_{1,2} a_{2,2}.
	\qe
\end{proof}
\begin{proof}[Proof of Proposition \ref{Prop:NonAsymptFree} and Proposition \ref{Prop:NonAsymptFree2} 1.] We expand the quantity
	\eq
		\lefteqn{ \Phi  \Big (  \big( x_1^2 - \Phi(x_1^2) \big) \big( y^2 - \Phi(y^2) \big)   \big( x_1^2 - \Phi(x_1^2) \big) \big( y^2 - \Phi(y^2) \big)  \Big)} \\
	& = & \Phi( x_1^2 y^2 x_1^2 y^2) \\
	& & - 2 \Phi(x_1^2) \Phi(x_1^2 y^4) - 2 \Phi(y^2) \Phi(y^2 x_1^4)\\
	& & + 4 \Phi(x_1^2) \Phi(y^2) \Phi(x_1^2 y^2) + \Phi(x_1^2)^2\Phi(y^4) + \Phi(y^2)^2 \Phi(x_1^4)\\
	& & - 4 \Phi(x_1^2)^2 \Phi(y^2)^2\\
	& & + \Phi(x_1^2)^2 \Phi(y^2)^2
	\qe
	
\noindent We have to compute $\Phi(x_1^2)$, $\Phi(x_1^2 y^2)$, $\Phi(x_1^4)$, $\Phi(x_i^4 y^2)$ and $\Phi(x_1^2 y^2 x_1^2 y^2)$. Using again the basic properties of Corollary \ref{Cor:BasicProp} and the computation of $\Phi(x_i^4)$ of Figure \ref{CyclesFat}, the only term we have to compute is $\Phi(x_1^2 y^2 x_1^2 y^2)$.
\begin{figure}[!h]
\begin{center}
 \includegraphics[width=150mm]{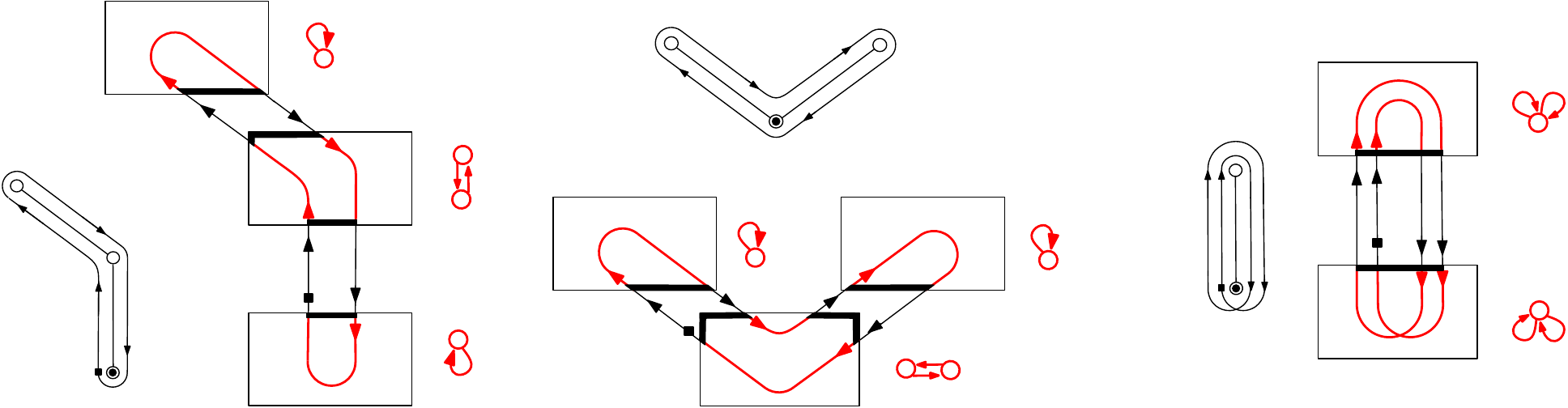}
 \caption{Computation of $\Phi \big( x_1P_1(\mathbf y) \dots x_1 P_4(\mathbf y) \big )$. We consider the enumeration of cycles of length 4 running on a tree of Figure \ref{CyclesFat}, and them open boxes on vertices.}
 \label{Computationmoments3}
 \end{center}
\end{figure}
\\
\\In Figure \ref{Computationmoments3}, we have computed $\Phi \big( x_1P_1(\mathbf y) \dots x_1 P_4(\mathbf y) \big )$ for any $^*$-polynomials $P_1\etc P_4$ and obtained
	\eq
	\Phi \big( x_1P_1(\mathbf y) \dots x_1 P_4(\mathbf y) \big ) & = &	a_{1,1}^2 \Phi \big( P_1(\mathbf y) P_3(\mathbf y) \big) \Phi \big( P_2(\mathbf y) \big) \Phi \big( P_4(\mathbf y) \big)  \\
		&  & + a_{1,1}^2 \Phi \big( P_1(\mathbf y) \big)  \Phi \big( P_2(\mathbf y) P_4(\mathbf y) \big) \Phi \big( P_3(\mathbf y) \big) \\
		& & + a_{1,2} \Phi^{(2)}\big( P_1(\mathbf y), P_3(\mathbf y)\big)   \Phi^{(2)}\big( P_2(\mathbf y), P_4(\mathbf y)\big).
	\qe

\noindent Specifying $P_1=P_3=1$ and $P_2=P_4=y^2$, we get
	\eq
		 \Phi( x_1^2 y^2 x_1^2 y^2) = a_{1,1}^2 \Phi(y^2)^2 + a_{1,1}^2 \Phi(y^4) + a_{1,2} \Phi^{(2)}(y^2,y^2).
	\qe
We then obtain
	\eq
		\lefteqn{ \Phi  \Big (  \big( x_1^2 - \Phi(x_1^2) \big) \big( y^2 - \Phi(y^2) \big)   \big( x_1^2 - \Phi(x_1^2) \big) \big( y^2 - \Phi(y^2) \big)  \Big)} \\
	& = & a_{1,1}^2 \Phi(y^2)^2 + a_{1,1}^2 \Phi(y^4) + a_{1,2} \Phi^{(2)}(y^2,y^2)\\
	& & - 2 a_{1,1}^2 \Phi(y^4) - 2  \Phi(y^2)^2 ( 2 a_{1,1}^2 + a_{1,2})\\
	& & + 4 a_{1,1}^2 \Phi(y^2)^2 + a_{1,1}^2\Phi(y^4) + \Phi(y^2)^2 ( 2 a_{1,1} + a_{1,2})\\
	& & - 3 a_{1,1}^2 \Phi(y^2)^2\\
	& = & a_{1,2} \big( \Phi^{(2)}(y^2, y^2) -  \Phi(y^2)^2 \big) =  a_{1,2}  \Phi^{(2)}\big( y^2 - \Phi(y^2), y^2 - \Phi(y^2) \big),
	\qe
	
\noindent where in the last equality we have used the bi-linearity of $\Phi^{(2)}$.
\end{proof}

\begin{proof}[Proof of Proposition \ref{Prop:NonAsymptFree} 2.] If $Y_N$ is a diagonal matrix, then it turns out that $\Phi^{(2)}(P,Q) = \Phi(PQ)$ for any $^*$-polynomials $P,Q$, and the claim follows directly.
\end{proof}

\section[Proof of the Schwinger Dyson equations]{Proof of the Schwinger Dyson system of equations for heavy Wigner and diagonal matrices, Theorem 2.12}

\noindent The idea of the proof of Theorem \ref{th:SchwingerDyson} is to classify, in the combinatorial approach by cycles running on trees, those for which the cycle visits a fixed number of times the first edge of the tree. 
\\
\\Before providing Theorem \ref{th:SchwingerDyson}, we apply theses equations to give an other computation of $\Phi[x_1^2x_2^2x_1^2x_2^2]$. First, we enumerate the decompositions
\begin{eqnarray*}
	x_1^2x_2^2x_1^2x_2^2 & = & (x_1 \times 1 \times x_1) x_2^2x_1^2x_2^2\\
		& = & (x_1  \times x_1x_2^2 \times x_1) x_1x_2^2\\
		& = & (x_1 \times  x_1x_2^2x_1 \times x_1) x_2^2\\
		& = & (x_1  \times 1 \times x_1) x_2^2 (x_1  \times 1 \times x_1) x_2^2.
\end{eqnarray*}
Then, by Theorem \ref{th:SchwingerDyson} we get
\begin{eqnarray*}
	\Phi[ x_1^2x_2^2x_1^2x_2^2 ]  & = & a_{1,1}\Big( \Phi[1] \Phi[x_2^2x_1^2x_2^2] + \Phi[x_1x_2^2]\Phi[ x_1x_2^2] + \Phi[ x_1x_2^2x_1]\Phi[ x_2^2] \Big)\\
		& & + a_{1,2} \Phi^{(2)}(1,1)\Phi^{(2)}(x_2^2,x_2^2)\\
		& = &  a_{1,1}\Big( \Phi[x_1^2] \Phi[x_2^4] + 0 + \Phi[ x_1^2]\Phi[x_2^2]^2  \Big)+ a_{1,2}\Phi^{(2)}(x_2^2,x_2^2)\\
		& = &  a_{1,1}^2a_{2,1}^2 + a_{1,1}^2 \Phi[x_2^4]  + a_{1,2}\Phi^{(2)}(x_2^2,x_2^2),
\end{eqnarray*}
where we have used the facts that $\Phi[x_1^nx_2^m]=\Phi[x_1^n] \Phi[x_2^m]$ for any $n,m\geq 1$ and that $\Phi[x_i^2]=a_{i,1}$ for $i=1,2$. By Theorem \ref{th:SchwingerDyson}, one has with a similar computation
\begin{eqnarray*}
	\Phi[x_2^4] & = & a_{2,1} \Big(  \Phi[1]\Phi[x_2^2] + \Phi[x_2]\Phi[x_2] + \Phi[x_2]\Phi[1] \Big)+a_{2,2}  \Phi^{(2)}(1,1) \Phi^{(2)}(1,1)\\
		& = & 2 a_{2,1}^2 +a_{2,2}.
\end{eqnarray*}
To compute $\Phi^{(2)}(x_2^2,x_2^2)$ with Theorem \ref{th:SchwingerDyson}, we enumerate the decompositions
		$$(x_2^2,x_2^2) = \Big( (x_2 \times 1\times x_2)1, x_2^2 \Big) =  \Big( (x_2 \times 1\times x_2)1,1(x_2 \times 1\times x_2)1\Big).$$
So we have
\begin{eqnarray*}
	\Phi^{(2)}(x_2^2,x_2^2) & = & a_{2,1}\Phi[1] \Phi^{(2)}(1,x_2^2) + a_{2,2}\Phi^{(2)}(1,1)\Phi^{(3)}(1,1,1)\\
		& = & a_{2,1}^2 +a_{2,2}.
\end{eqnarray*}
We then get as expected
\begin{eqnarray*}
	\Phi[ x_1^2x_2^2x_1^2x_2^2 ] & = & a_{1,1}^2a_{2,1}^2 + a_{1,1}^2(2a_{2,1}^2+a_{2,2}) + a_{1,2}(a_{2,1}^2+a_{2,2})\\
		& = & 3a_{1,1}^2a_{2,1}^2 + a_{1,1}^2a_{2,2} + a_{1,2}a_{2,1}^2 + a_{1,2}a_{2,2}.
\end{eqnarray*}

\begin{proof}[Proof of Theorem \ref{th:SchwingerDyson}] For clarity of the exposition, we start by proving (\ref{TraceFormEq2}) for $K=1$, that is: for any $j=1\etc p$ and any monomial $P$,
\begin{equation}\label{TraceFormEq1}
		\Phi(x_jP) = \sum_{k\geq 1} a_{j,k}\sum_{ x_jP = (x_jL_1x_j)R_1 \dots (x_jL_kx_j)R_k } \Phi^{(k)}  ( L_1 \etc L_k \big ) \Phi^{(k)}  ( R_1 \etc R_k \big ).
\end{equation}

\noindent We write $P = P_1(\mathbf y) \times x_{\gamma(2)} P_2(\mathbf y) \dots x_{\gamma(L)} P_L(\mathbf y)$ and set $\gamma(1)=j$. By Proposition \ref{Prop:ComputMoments},
	\eq
		\Phi(x_jP) = \sum_{(G,c) \in \mathcal L^{(\gamma)}}   \omega_{HW}(G, c) \times  \omega_{TR}(G, c),
	\qe
where $\omega_{HW}$ and $\omega_{TR}$ are given in Definitions \ref{def:HWWeight} and \ref{def:DWeight}.
\\
\\{\it Step 1: Cycle visiting $2K$ times the first edge}

\noindent Let $(G,c)$ in $\mathcal L^{(\gamma)}$. The root of $G$ is called the vertex number $1$, the second vertex visited by $c$ is called the number $2$. Saying that the undirected edge $\{1,2\}$ is visited exactly $2K$ times is equivalent to say that
\begin{enumerate}
\item there exist cycles $ d^{(1)} \etc d^{(K)}$ starting at the vertex $2$,
\item there exist cycles $ e^{(1)} \etc e^{(K)}$ starting at the vertex $1$,
\item theses cycles do not visit $\{1,2\}$,
\item $c$ can be written
\begin{equation}\label{SDProofDecC}
		c =  a \circ d^{(1)} \circ a^*  \circ e^{(1)} \circ a \circ d^{(2)} \circ a^* \circ e^{(2)} \circ   \dots  \circ a  \circ d^{(K)} \circ  a^*  \circ e^{(K)},
\end{equation}
\end{enumerate}
where $\circ$ denotes the composition of paths, $a=(1,2)$ and $a^* = (2,1)$. See Figure \ref{fig:Schwinger} for an example.

\begin{figure}[!h]
\begin{center}
\includegraphics[height=70mm]{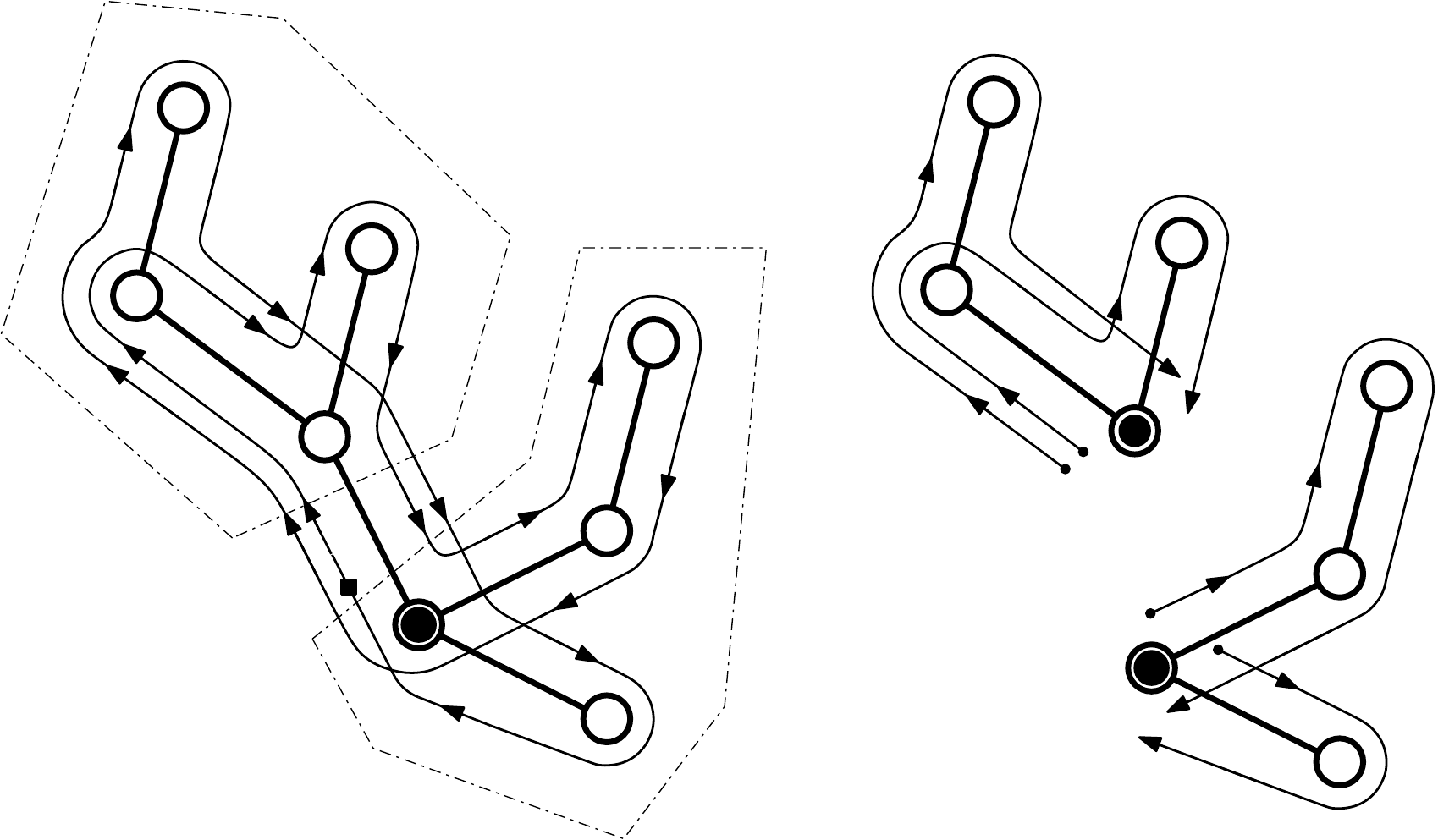}\caption{Left: the cycle visits $4$ times the first vertex. Right: the two couples of cycles induces.}\label{fig:Schwinger}
\end{center}
\end{figure}

\noindent Assume that $c$ is of this form. Since the edge $\{1,2\}$ can only be visited by steps of color $\gamma(1)$, we can write 
\begin{equation}\label{SDProofDecL}
		\big(\gamma(1) \etc \gamma(L)\big)=(\gamma(1), \alpha^{(1)},  \gamma(1), \beta^{(1)} , \gamma(1), \alpha^{(2)},\gamma(1), \beta^{(2)} \etc \gamma(1), \alpha^{(K)},\gamma(1), \beta^{(K)}),
\end{equation}
where for any $k=1\etc K$, $\alpha^{(k)}=\big( \alpha^{(k)}(1) \etc \alpha^{(k)} (L_{\alpha,k}) \big)$ and $\beta^{(k)}=\big(\beta^{(k)}(1) \etc \beta^{(k)}(L_{\beta,k}) \big)$ are families of integers in $\{1 \etc p\}$ for any $k=1\etc K$. The integer $ L_{\alpha,k} $ is the length of the cycle $\alpha^{(k)}$ and $L_{\beta,k}$ is the length of the cycle $\beta^{(k)}$.
\\
\\The two cycles $d   =   d^{(1)} \circ \dots \circ d^{(K)}$ and $e = e^{(1)} \circ \dots \circ e^{(K)}$ are induced by $c$ on disjoint subgraphs $G_d$ and $G_e$ of $G$ respectively. For any $k=1\etc K$, the cycle $d^{(k)}$ has steps of color $\alpha^{(k)}$, so that the steps of $d$ are colored by $\alpha = (\alpha^{(1)} \etc \alpha^{(K)})$. The same holds for $e$ with $\beta = (\beta^{(1)} \etc \beta^{(K)})$.
\\
\\Hence, by rooting $G_d$ on $2$ and $G_e$ on $1$, we get that $(G_d,d)$ belongs to $\mathcal L^{(\alpha)}$ and $(G_e,e)$ belongs to $\mathcal L^{(\beta)}$. They are not typical elements of these sets, in the sense that for any $k\geq 1$, the cycle $d$ always comes back to the root of its tree after $L_{\alpha,1}+ \dots +L_{\alpha,k}$ steps. Hence the following definition.

\begin{Def}[Chain of cycles running on trees]~\label{Def:ChainCycles}
\\Let $K\geq 1$ be an integer, $\mathbf L= (L_1 \etc L_K)$ be a family of integers and $\gamma : \{1\etc L\} \to \{1\etc p\}$, where $L=L_1 + \dots +L_K$. We set $\mathcal  L^{(\gamma)}_{\mathbf L}$ the set of all couples $(L,c)$ in $\mathcal  L^{(\gamma)}$, such that $c$ can be written $c = c_1 \circ \dots \circ c_K$, where for any $k=1\etc K$ the cycle $c_k$ is of length $L_k$.
\end{Def}

\noindent If we denote $\mathbf L_\alpha = (L_{\alpha,1} \etc L_{\alpha,K}) $ and $\mathbf L_\beta= (L_{\beta,1} \etc L_{\beta,K}) $, then we get that actually $(G_d,d)$ belongs to $\mathcal L^{(\alpha)}_{\mathbf L_\alpha}$ and that $(G_e,e)$ belongs to $\mathcal L^{(\beta)}_{\mathbf L_\beta}$.
\\
\\{\it Step 2: Reciprocal construction}
\\Let $K\geq 1$ be an integer and consider a decomposition
\begin{equation}\label{SDProofDecLRec}
		\big(\gamma(1) \etc \gamma(L)\big)=(\gamma(1), \alpha^{(1)},  \gamma(1), \beta^{(1)} , \gamma(1), \alpha^{(2)},\gamma(1), \beta^{(2)} \etc \gamma(1), \alpha^{(K)},\gamma(1), \beta^{(K)}),
\end{equation}
where for any $k=1\etc K$ one has $\alpha^{(k)}$ is in $\{1 \etc p\}^{L_{\alpha,k}}$ and $\beta^{(k)}$ in $\{1 \etc p\}^{L_{\beta,k}}$ for sequences of integers $\mathbf L_\alpha= (L_{\alpha,1} \etc L_{\alpha,K}) $ and $\mathbf L_\beta = (L_{\beta,1} \etc L_{\beta,K} ) $. Define $\alpha= (\alpha^{(1)} \etc \alpha^{(K)})$ and $\beta= (\beta^{(1)} \etc \beta^{(K)})$. Let $(G_d,d)$ in $\mathcal L^{(\alpha)}_{\mathbf L_\alpha}$ and $(G_e,e)$ in $\mathcal L^{(\beta)}_{\mathbf L_\beta}$. We write $d =  d^{(1)} \circ \dots \circ d^{(K)}$ and $e = e^{(1)}  \circ \dots \circ e^{(K)}$ with the notation for chains of cycles (Definition \ref{Def:ChainCycles}). As we embed the graphs $G_d$ and $G_e$, link their roots by an new edge, extend the cycles into a cycle $c_{d,e}$ as in (\ref{SDProofDecC}) and set the root of $G_{d,e}$ to be the root of $G_{e}$, we get an element $(G_{d,e},c_{d,e})$ in $\mathcal L^{(\gamma)}$. 
\\
\\As $(G_d,d)$ and $(G_e,e)$ are the chains of cycles running on a tree of the previous step when we start with $(G_{d,e},c_{d,e})$, we have proved that
	\eqa\label{eq:SplitCycle}
		\Phi(x_jP)  = \sum_{K\geq 1} \   \sum_{  \substack{ (\alpha^{(1)} \etc \alpha^{(K)})\\ (\beta^{(1)}\etc \beta^{(K)}) \\ \textrm{ as in (\ref{SDProofDecL}) }}} \ \sum_{  \substack{ (G_d,d ) \in \mathcal L^{(\alpha)}_{\mathbf L_\alpha }  \\ (G_e,e) \in \mathcal L^{(\beta)}_{\mathbf L_\beta }} } \omega_{HW} (G_{d,e},c_{d,e} ) \times \omega_{TR} (G_{d,e},c_{d,e}  ).
	\qea

\noindent{\it Step 3: Computation of the weights}
\\We have obviously $\omega_{HW} (G_{d,e},c_{d,e} ) = a_{j,K} \ \omega_{HW} (G_d,d)  \times \omega_{HW} (G_e,e )$. For the weight $\omega_{TR}$, it is important to take care about the dependence on the polynomials involved. We then write $ \omega_{TR} =  \omega_{TR}(P)$ in the notation of Definition \ref{def:DWeight}. In the case considered in that proof, the polynomial has been denoted $x_jP$. Recall that 
	\eqa
		 \omega_{TR}(x_jP)(G_{d,e}, c_{d,e}) = \prod_{ v \textrm{ vertex of } G_{d,e}}  \tau[T_v],
	\qea

\noindent where the $^*$-graph tests are obtained by opening boxes on the vertices of $G_{d,e}$ as in Figure \ref{fig:NodesToTestGraphs3}. Since the matrices $\mathbf Y_N$ are diagonal, we actually have a much simpler expression for this weight. First, remark that the diagonality implies that for any vertex $v$ of $G_{d,e}$, $\tau[T_v] = \tau[ T_v^\pi]$ where $\pi$ is the partition of the vertices of $T_v$ with only one block. Hence, this quantity is equal to $\Phi\big( \prod_{n} P_n(\mathbf y) \big)$, where the product is over all integers $n=1\etc L$ such that the $n$-th step of $c_{d,e}$ is incident in $v$. For $v$ different that the roots of $G_d$ and $G_e$, this quantity is the same as if we replace $c_{d,e}$ by $c_d$ or $c_e$ (depending if the vertex comes from $G_d$ or $G_e$). For $v$ the root of $G_d$ or $G_e$, we have to take into account the steps of $c_{d,e}$ on the edge $\{1,2\}$.
\\
\\Given $(d,e)$ as in the sum (\ref{eq:SplitCycle}), we get a decomposition of the polynomial $x_jP$:
	\eq
		 x_jP = (x_jL_1x_j)R_1 \dots (x_j L_K x_j)R_K,
	\qe

\noindent where the position of the $x_j$'s corresponds to the position of the $\gamma(1)$'s in (\ref{SDProofDecLRec}), where we have decomposed $\big(\gamma(1) \etc \gamma(K) \big)$. As $x_jP =  x_{\gamma(1)} P_1(\mathbf y) \times x_{\gamma(2)} P_2(\mathbf y) \times  \dots  \times x_{\gamma(L)} P_L(\mathbf y)$, we can write $L_k = P_{i_k}(\mathbf y) \tilde L_k P_{j_k-1}(\mathbf y)$ and $R_k = P_{j_k}(\mathbf y) \tilde R_k P_{i_{k+1}-1}(\mathbf y)$ for any $k=1\etc K$. The integers $i_k$'s corresponding to the steps where the cycle runs on $(1,2)$, the integers $j_k$ to the steps where it runs on $(2,1)$.
\\
\\We set $\bar L_k = \tilde L_k P_{j_k-1}(\mathbf y) P_{i_k}(\mathbf y) $, $\bar R_k = \tilde R_k P_{i_{k+1}-1}(\mathbf y)P_{j_k}(\mathbf y)$ and then $\bar L = \bar L_1 \dots \bar L_K$ and $\bar R = \bar R_1 \dots \bar R_K$. As the polynomials $P_{j_k-1}(\mathbf y)$ and $P_{i_k}(\mathbf y)$ (respectively $ P_{i_{k+1}-1}(\mathbf y)$ and $P_{j_k}(\mathbf y)$) contribute in the root of $G_d$ (respectively $G_e$), we get
	\eqa
		 \omega_{TR}(x_jP)(G_{d,e}, c_{d,e}) = \omega_{TR}[\bar L](G_{d}, c_{d})  \times  \omega_{TR}[\bar R](G_{e}, c_{e}).
	\qea

\noindent {\it Step 4: Conclusion of the combinatorial decomposition}
\\We have obtained that
	\eqa
		\nonumber \Phi(x_jP) & =& \sum_{K\geq 1} \  a_{j,K} \sum_{  \substack{ \alpha, \beta \\ \textrm{ as in (\ref{SDProofDecL})} }} \bigg( \sum_{  (G_d,d ) \in \mathcal L^{(\alpha)}_{\mathbf L_\alpha }  } \omega_{HW} (G_{d},c_{d} ) \times \omega_{TR}(L) (G_{d},c_{d}  )\\
			& &  \times \sum_{(G_e,e) \in \mathcal L^{(\beta)}_{\mathbf L_\beta }}\omega_{HW} (G_{e},c_{e} ) \times \omega_{TR}(R) (G_{e},c_{e}  ) \bigg). \label{eq:CombForm}
	\qea
	
\noindent The classification in the sum over $\alpha, \beta$ is in correspondence with the number of way we can decompose $x_jP$ into
	\eq
		 x_jP = (x_jL_1x_j)R_1 \dots (x_j L_K x_j)R_K.
	\qe
	
\noindent It remains to interpret the combinatorial terms in terms of the multilinear forms $(\Phi^{(K)})_{K\geq 1}$.

\begin{Prop}[Proposition \ref{Prop:ComputMoments} continued]~\label{Prop:ComputMoments2}
\\For any polynomial $P_1 \etc P_K$ of the form $P_k = x_{\gamma_k(1)} P_{k,1}(\mathbf y) \dots x_{\gamma_k(L_k)} P_{L_k}(\mathbf y), \ k \geq 1$, one has
	\eqa
		\Phi^{(K)}(P_1 \etc P_K) = \sum_{(G,c) \in \mathcal L^{(\gamma)}_{\mathbf L}}   \omega_{HW}(G, c) \times  \omega_{TR}(P)(G, c),
	\qea
where $\mathbf L = (L_1 \etc L_K)$, $\gamma = (\gamma_1 \etc \gamma_K)$ seen as a map $\{1 \etc L_1 + \dots + L_K\} \to \{1\etc p\}$, and $P = P_1 \dots P_K$. The weights $\omega_{HW}$ and $\omega_{TR}$ are the same as in proposition \ref{Prop:ComputMoments}.
\end{Prop}

\begin{proof}[Proof of Proposition \ref{Prop:ComputMoments2}] Based on formula (\ref{eq:ApplicFree}) given by the traffic asymptotic freeness of $\mathbf X_N$ and $\mathbf Y_N$, this proposition is obtained by a minor modification of the proof of Proposition \ref{Prop:ComputMoments}.
\\
\\Denote by $T_{P_1} \etc T_{P_K}$ and $T_P$ the $^*$-test graphs obtained as in Figure \ref{SimpleCycle2} for polynomials $P_1 \etc P_K$ and $P$ respectively. Consider the $^*$-test graph $T$ obtained by merging the roots of $T_{P_1} \etc T_{P_K}$. Then, $\Phi^{(K)}(P1 \etc P_K) = \tau[T]$, where $\tau$ is the limiting distribution of traffics of $(\mathbf X_N, \mathbf Y_N, \mathbf Y_N^*)$.
\\
\\Let $\pi_0$ be the partition of $\{1\etc 2(L_1 + \dots + L_K)\}$ that puts in a same block $\{1, 2L_1+1, 2(L_1+L_2)+1 \etc 2(L_1+ \dots + L_{K-1})+1 \}$ and let alone the other integers. Remark that $T = T_P^{\pi_0}$. By (\ref{eq:ApplicFree}), one has
	\eq
		\Phi^{(K)}(P_1 \etc P_K)  & = &  \sum_{\pi \in \mathcal P(2(L_1 + \dots + L_K))} \mathbf 1_{( T^\pi \textrm{ is a free product} )} \prod_{ \tilde T }   \tau^0 [ \tilde T ] \\
		& = & \sum_{\pi \leq \pi_0} \mathbf 1_{( T_P^\pi \textrm{ is a free product} )} \prod_{ \tilde T }   \tau^0 [ \tilde T ],
	\qe

\noindent where $\pi \leq \pi_0$ means that $\pi$ is a sub-partition of $\pi_0$. We find the result with the same reasoning as in Proposition \ref{Prop:ComputMoments}, as we realize that the condition $\pi \leq \pi_0$ exactly means that the cycle comes back at the root after $2L_1, 2L_2, \dots$ and $2L_{K-1}$ steps.
\end{proof} 

\noindent By (\ref{eq:CombForm}) and Proposition, we get 
	\eq
		\Phi(x_jP) & =& \sum_{K\geq 1} \  a_{j,K} \sum_{    x_jP = (x_jL_1x_j)R_1 \dots (x_j L_K x_j)R_K }   \Phi^{(K)}( \bar L_1 \etc \bar L_K) \times \Phi^{(K)}( \bar R_1 \etc \bar R_K).
	\qe
	
\noindent As the matrices $\mathbf Y_N$ are diagonal, for any polynomials $P_1\etc P_K$, any polynomial $Q(\mathbf y)$ and any $k=1\etc K$, one has 
	\eq
		\Phi^{(K)}\big( P_1\etc P_{k-1}, P_kQ(\mathbf y ), P_{k+1} \etc P_K) =   \Phi^{(K)}\big( P_1\etc P_{k-1},Q(\mathbf y ) P_k, P_{k+1} \etc P_K).
	\qe

\noindent Hence we get as expected
	\eq
		\Phi(x_jP) & =& \sum_{K\geq 1} \  a_{j,K} \sum_{    x_jP = (x_jL_1x_j)R_1 \dots (x_j L_K x_j)R_K }   \Phi^{(K)}(   L_1 \etc   L_K) \times \Phi^{(K)}( R_1 \etc  R_K).
	\qe
	
\noindent{\it Step 5: the general case}
\\Let $P_1\etc P_K$ be monomials of the form
	\eq
		P_1 & = & P_{1,1}(\mathbf y) \times x_{\gamma_1(2)} P_{1,2}(\mathbf y) \dots  x_{\gamma_1(L_1)} P_{1,L_1}(\mathbf y)\\
		P_k & = & x_{\gamma_k(1)} P_{k,1}(\mathbf y) \dots  x_{\gamma_k(L_k)} P_{k,L_k}(\mathbf y), \ k =2 \etc K.
	\qe

\noindent By setting $\gamma$ and $\mathbf L$ as in Proposition \ref{Prop:ComputMoments2}, we have 
	\eqa
		\Phi^{(K)}(x_jP_1 \etc P_K) = \sum_{(G,c) \in \mathcal L^{(\gamma)}_{\mathbf L}}   \omega_{HW}(G, c) \times  \omega_{TR}(P)(G, c).
	\qea

\noindent Let $(G,c)$ in $ \mathcal L^{(\gamma)}_{\mathbf L}$, and write $c$ as a composition of cycles of length $L_1 \etc L_K$, namely $c=c_1 \circ \dots c_K$. Saying that $c$ visits $\{1,2\}$ exactly $2k$ times is equivalent to say there exist non negative integers $s_1\etc s_K$ such that 
\begin{itemize}
	\item $s_1\geq 1$, $s_2\etc s_K\geq 0$,
	\item $s_1 + \dots+s_{K}=k$,
	\item for any $m=1\etc K$, the cycle $c_m$ visits $a$ exactly $2s_m$ times.
\end{itemize}
Assume that for any $m=1\etc K$, the cycle $c_m$ visits $\{1,2\}$ exactly $2s_m$ times. Then we get a decomposition 
		$$c_1 =  a \circ d^{(1,1)} \circ a^*  \circ e^{(1,1)} \circ a \circ d^{(1,2)} \circ a^* \circ e^{(1,2)} \circ   \dots  \circ a  \circ d^{(1,s_1)} \circ  a^*  \circ e^{(1,s_1)},$$
and for any $m=2\etc K$,
	$$c_m =  e^{(m,0)} \circ a \circ d^{(m,1)} \circ a^*  \circ e^{(m,1)} \circ a \circ d^{(m,2)} \circ a^* \circ e^{(m,2)} \circ   \dots  \circ a  \circ d^{(m,s_1)} \circ  a^*  \circ e^{(m,s_1)}.$$
The only difference is that the cycles $c_1\etc c_m$ are not constrained to visit $\{1,2\}$ during their first step. The rest of the proof can be written as we made for the proof of (\ref{TraceFormEq1}), without any new niceties. 
We the same reasoning as before, we obtain the expected result, i.e. Theorem \ref{th:SchwingerDyson}.
\end{proof}

\section[Proof of the characterization of the spectrum]{Proof of Proposition 2.13: a characterization of the spectrum of a single heavy Wigner matrix}

\noindent We manipulate truncated sums. We write computations based on Theorem \ref{th:SchwingerDyson} for the truncation
	\eq
		F^\lambda_N(K) : = \frac 1 {\lambda^K} \ \sum_{n=0}^N \ \sum_{n_1 + \dots + n_K =n} \ \frac 1 {\lambda^{n}}\Phi^{(K)}(x^{n_1+1},x^{n_2}\etc x^{n_K})
	\qe

\noindent of the formal power series associated to
	\eq
		\Phi^{(K)}\big( x (\lambda-x)^{-1}, (\lambda-x)^{-1} \etc (\lambda-x)^{-1} \big).
	\qe
\noindent Remark that it is equal in the sense of formal sums to 
	\eq
		\lambda G^\lambda(K) -G^\lambda(K-1),
	\qe
\noindent and then it enough to prove 
	\eq
		F^\lambda_N(K) \limN \frac 1 \lambda \bigg( \sum_{k\geq 1} a_k \binom{K+k-2}{K-1}G^{\lambda}  (k) G^{\lambda}  (k+K-1) \bigg),
	\qe

\noindent in the convergence of power series in $\frac 1 \lambda$. We fix an integer $N\geq 1$. By Theorem \ref{th:SchwingerDyson}, we have
	\eq
		F^\lambda_N(K)  = \lefteqn{ \frac 1 {\lambda^K} \ \sum_{n=0}^N \ \sum_{n_1 + \dots + n_K =n} \ \frac 1 {\lambda^{n}}\Phi^{(K)}(x^{n_1+1},x^{n_2}\etc x^{n_K})}\\
	& = & \frac 1 {\lambda^K} \sum_{n=0}^N \ \sum_{n_1 + \dots + n_K =n} \ \frac 1 {\lambda^n} \sum_{1\leq k\leq \frac {n+1}2 } a_k \\
	&    &  \times \ \ \  \sum_{   \substack { s_1 + \dots + s_K=k  \\ 1\leq  s_1\leq \frac {n_1+1}2  \\ 0\leq s_2\leq \frac {n_2}2 \etc 0\leq s_K\leq \frac {n_K}2}} \ \sum_{(\mathbf r, \mathbf t)} \Phi^{(k)}(x^{\mathbf r})\Phi^{(k+K-1)}(x^{\mathbf t}).
\end{eqnarray*}
The last sum is over all families of non negative integers
		$$\mathbf r = (r^{(1)}_1 \etc r^{(1)}_{s_1}\etc r^{(K)}_1 \etc r^{(K)}_{s_K}),$$
		$$\mathbf t = (t^{(1)}_1 \etc t^{(1)}_{s_1}, t^{(2)}_0 \etc t^{(2)}_{s_2}\etc t^{(K)}_0 \etc t^{(K)}_{s_K}),$$
such that
\begin{eqnarray*}
		r^{(1)}_1+\dots +r^{(1)}_{s_1} + t^{(1)}_1+\dots +t^{(1)}_{s_1} & = &  n_1+1-2s_1,\\
		r^{(i)}_1+\dots +r^{(i)}_{s_i} + t^{(i)}_0+\dots +t^{(i)}_{s_i} & = & n_i-2s_i, \ i=2 \etc K.
\end{eqnarray*}
We have used (and we will use) the notation 
		$$\Phi^{(k)}(x^{\mathbf r}) = \Phi^{(k)}( x^{r^{(1)}_1} \etc x^{r^{(1)}_{s_1}}\etc x^{r^{(K)}_1} \etc x^{r^{(K)}_{s_K}}).$$ The restrictions on the third and fourth sums follow from consideration on the degree on the monomials we compute. Now we interchange the order of summation of $(n_1 \etc n_K)$ and $(s_1 \etc s_K)$.
\begin{eqnarray*}
\lefteqn{ \frac 1 {\lambda^K} \ \sum_{n=0}^N \sum_{n_1 + \dots + n_K =n} \ \frac 1 {\lambda^{n}}\Phi^{(K)}(x^{n_1+1},x^{n_2}\etc x^{n_K})}\\
	& = & \frac 1 {\lambda^K} \sum_{1\leq k \leq \frac {N+1}2 } a_k \ \sum_{\substack{ s_1 + \dots + s_K =k \\ s_1\geq 1, \ s_2\etc s_K\geq 0}} \ \sum_{2k-1 \leq n \leq N }  \ \frac 1 {\lambda^n} \\
	&   & \ \ \ \times   \sum_{   \substack { n_1 + \dots + n_K=n  \\  n_1\geq 2s_1-1 \\ n_2\geq 2s_2 \etc n_K\geq 2s_K }} \ \sum_{\mathbf l} \sum_{\mathbf r } \Phi^{(k)}(x^{\mathbf r}) \sum_{\mathbf t} \Phi^{(k+K-1)}(x^{\mathbf t}).
\end{eqnarray*}
By the sum over $\mathbf l$, we mean the sum over all families of non negative integers $\mathbf l=(l_1 \etc l_K)$ such that
\begin{eqnarray*}
		0 \leq & l_1 & \leq n_1 +1-2s_1,\\
		0 \leq & l_2 & \leq n_2 -2s_2,\\
		\vdots &  & \\
		0 \leq & l_K & \leq n_K - 2s_K.
\end{eqnarray*}
By the sum over $\mathbf r$, we mean the sum over all families of non negative integers 
		$$\mathbf r = (r^{(1)}_1 \etc r^{(1)}_{s_1}\etc r^{(K)}_1 \etc r^{(K)}_{s_K}),$$
such that
\begin{eqnarray*}
		r^{(1)}_1+\dots +r^{(1)}_{s_1} & = &  l_1,\\
		r^{(2)}_1+\dots +r^{(2)}_{s_2} & = & l_2,\\
		\vdots & & \\
		r^{(K)}_1+\dots +r^{(K)}_{s_K} & = & l_K.
\end{eqnarray*}
At last, by the sum over $\mathbf t$, we mean the sum over all families of non negative integers 
		$$\mathbf t = (t^{(1)}_1 \etc t^{(1)}_{s_1}, t^{(2)}_0 \etc t^{(2)}_{s_2}\etc t^{(K)}_0 \etc t^{(K)}_{s_K}),$$
such that
\begin{eqnarray*}
		 t^{(1)}_1+\dots +t^{(1)}_{s_1} & = &  n_1+1-2s_1-l_1,\\
		 t^{(2)}_0+\dots +t^{(2)}_{s_2} & = & n_2-2s_2-l_2,\\
		\vdots & & \\
		t^{(K)}_0+\dots +t^{(K)}_{s_K} & = & n_k-2s_K-l_K.
\end{eqnarray*}
Given $k,s_1\etc s_2$ as in the previous formula, we set the change of variable for $n,n_1 \etc n_K$
\begin{eqnarray*}
		m & = & n+1-2k,\\
		m_1 & = & n_1+1-2s_1,\\
		m_2 & = & n_2 -2s_2,\\
		\vdots & & \\
		m_K & = & n_K -2s_K.\\
\end{eqnarray*}
Remark first that
		$$\frac 1 {\lambda^K} \times \frac 1 {\lambda^n} = \frac 1 {\lambda^m}\times  \frac 1 {\lambda^{k+K-1}} \times \frac 1 {\lambda^{k}}.$$
Hence we get
\begin{eqnarray*}
\lefteqn{ \frac 1 {\lambda^K} \ \sum_{n=0}^N \ \sum_{n_1 + \dots + n_K =n} \ \frac 1 {\lambda^{n}}\Phi^{(K)}(x^{n_1+1},x^{n_2}\etc x^{n_K})}\\
	& = & \sum_{1\leq k \leq \frac {N+1}2 } a_k   \sum_{\substack{ s_1 + \dots + s_K =k \\ s_1\geq 1, \ s_2\etc s_K\geq 0}}  \sum_{ m = 0}^{N+1-2k }    \frac 1 {\lambda^m} \\
	&    & \ \ \ \ \times  \sum_{ m_1 + \dots + m_K=m}   \sum_{\substack{ l_1 =0\dots m_1 \\ \dots \\ l_K=0 \dots m_K}} \sum_{\mathbf r } \frac 1 {\lambda^k}\Phi^{(k)}(x^{\mathbf r}) \sum_{\mathbf t}  \frac 1 {\lambda^{k+K-1}}\Phi^{(k+K-1)}(x^{\mathbf t}).
\end{eqnarray*}
The sum over $\mathbf r$ is the same as before, and now last sum is over all families of non negative integers 
		$$\mathbf t = (t^{(1)}_1 \etc t^{(1)}_{s_1}, t^{(2)}_0 \etc t^{(2)}_{s_2}\etc t^{(K)}_0 \etc t^{(K)}_{s_K}),$$
such that
\begin{eqnarray*}
		 t^{(1)}_1+\dots +t^{(1)}_{s_1} & = &  m_1-l_1,\\
		\vdots & & \\
		t^{(K)}_0+\dots +t^{(K)}_{s_K} & = & m_K-l_K.
\end{eqnarray*}
We replace the set variables $(m_1\etc m_K, l_1\etc l_K)$ by variables $p_1\etc p_K$ and $ q_1\etc q_K$ where for any $i=1\etc K$ we have set $p_i=m_i-l_i$ and $q_i=l_i$. Then we get
\begin{eqnarray*}
\lefteqn{ \frac 1 {\lambda^K} \ \sum_{n=0}^N \ \sum_{n_1 + \dots + n_K =n} \ \frac 1 {\lambda^{n}}\Phi^{(K)}(x^{n_1+1},x^{n_2}\etc x^{n_K})}\\
	& = & \sum_{1\leq k \leq \frac {N+1}2 } a_k  \  \sum_{ m = 0}^{N+1-2k }    \frac 1 {\lambda^m} \  \sum_{(\mathbf p, \mathbf q)} \  \sum_{\substack{ s_1 + \dots + s_K =k \\ s_1\geq 1, \ s_2\etc s_K\geq 0}} \\
	&     & \ \ \ \ \times   \sum_{\mathbf r} \ \frac 1 {\lambda^k}\Phi^{(k)}(x^{\mathbf r})  \ \sum_{\mathbf t}  \frac 1 {\lambda^{k+K-1}}\Phi^{(k+K-1)}(x^{\mathbf t}).
\end{eqnarray*}
The sum over $(\mathbf p, \mathbf q)$ is the sum over all families of non negative integers $\mathbf p=(p_1\etc p_K)$ and $\mathbf q= (q_1 \etc q_K)$ such that
		$$p_1+ \dots +p_K+q_1+\dots q_K =m.$$
The sum over $\mathbf r$ is the sum over all families of non negative integers 
		$$\mathbf r = (r^{(1)}_1 \etc r^{(1)}_{s_1}\etc r^{(K)}_1 \etc r^{(K)}_{s_K}),$$
such that
\begin{eqnarray*}
		r^{(1)}_1+\dots +r^{(1)}_{s_1} & = & q_1,\\
		\vdots & & \\
		r^{(K)}_1+\dots +r^{(K)}_{s_K} & = & q_K.
\end{eqnarray*}
The sum over $\mathbf t$ is the sum over all families of non negative integers 
		$$\mathbf t = (t^{(1)}_1 \etc t^{(1)}_{s_1}, t^{(2)}_0 \etc t^{(2)}_{s_2}\etc t^{(K)}_0 \etc t^{(K)}_{s_K}),$$
such that
\begin{eqnarray*}
		 t^{(1)}_1+\dots +t^{(1)}_{s_1} & = & p_1,\\
		 t^{(2)}_0+\dots +t^{(2)}_{s_2} & = & p_2,\\
		\vdots & & \\
		t^{(K)}_0+\dots +t^{(K)}_{s_K} & = &p_K.
\end{eqnarray*}
Let $K\geq 1$ and $k\geq 1$ be integers. Then there exist $\binom {K+k-2}{K-1}$ tuples of non negative integers $(s_1\etc s_K)$ such that $s_1+\dots s_K=k$, $s_1\geq 1$ and $s_2\etc s_K\geq 0$. Hence we get
\begin{eqnarray*}
\lefteqn{ \frac 1 {\lambda^K} \ \sum_{n=0}^N \ \sum_{n_1 + \dots + n_K =n} \ \frac 1 {\lambda^{n}}\Phi^{(K)}(x^{n_1+1},x^{n_2}\etc x^{n_K})}\\
	& = & \sum_{1\leq k \leq \frac {N+1}2 } a_k \ \binom {K+k-2}{K-1} \\
	& & \ \ \ \times   \sum_{0\leq p+q\leq N+1-2k} \ \frac 1 {\lambda^k}  \ \sum_{r_1+ \dots + r_k=q} \ \ \frac 1 {\lambda^q}\Phi^{(k)}(x^{r_1}\etc x^{r_q})  \\
	& & \ \ \ \times  \frac 1 {\lambda^{k+K-1}} \ \sum_{ t_1+\dots +t_{k+K-1}=p } \ \   \frac 1 {\lambda^{p}}\Phi^{(k+K-1)}(x^{t_1} \etc x^{t_{k+K-1}}).
\end{eqnarray*}
This gives the expected result by identification of the coefficients. The uniqueness of the solution of the equations follows directly from the observation of the valence of the formal power series.

\appendix

\section{On the possible parameters for a heavy Wigner matrix} \label{sec:Appendix}

\begin{Prop}\label{prop:NecessaryCondition} If a sequence $(a_k)_{k\geq 1}$ of real numbers is a parameter of a heavy Wigner matrix, then $(a_{k-1})_{k\geq 1}$ is the sequence of even moments of a Borel measure $m$ with finite moments, i.e. for any $k\geq 2$, $a_k=\int t^{2k-2}\mathrm{d}m(t)$. In particular, if the parameter $(a_k)_{k\geq 1}$ is non trivial then one has $a_k>0$ for any $k\geq 2$.
\end{Prop}

\noindent By the Hamburger's theorem \cite{HAM}, a sequence of real numbers $\big( \mu(k) \big)_{k\geq 1}$ is a sequence of moments if and only if, for any sequence $(x_k)_{k\geq 0}$ of complex numbers with finite support, one has 

\begin{equation}\label{eq:Hamburger}
		\sum_{j,k\geq 0} \mu(j+k) x_j \bar x_k\geq 0.
\end{equation}
		
\noindent Let $X_N$ be a heavy Wigner matrix of parameter $(a_k)_{k\geq 1}$ et let $\mu_N$ be the common law the sub-diagonal entries of $\sqrt N X_N$. As we do not change the parameter of a heavy Wigner matrices when we change sub-diagonal entries $X_{i,j}$ into $\eps | X_{i,j} |$, where $\eps$ is a random uniform sign, we can always assume that $\mu_N$ is real and symmetric for the task of the Proposition. Denote by $(\mu^{(N)}(k))_{k\geq 0}$ its sequence of moments. For any sequence $(y_k)_{k\geq 1}$ of complex numbers with finite support such that $y_0=0$, we apply (\ref{eq:Hamburger}) with $(x_k)_{k\geq 1}=(N^{-\frac k 2} y_k)_{k\geq 1}$: we get 

\begin{eqnarray*}
		\sum_{j,k\geq 0} \mu^{(N)}(j+k) x_j \bar x_k &= &N \sum_{j,k\geq 1} \frac {\mu^{(N)}(j+k)}{N^{ \frac {j+k}2-1}} y_j \bar y_k \\
		& = & N \sum_{j,k\geq 1} a_{\frac{j+k}2} y_{j} \bar y_{k} + o(1)\\
		& = & N \sum_{j,k\geq 0} a_{\frac{j+k}2+1} y_{j+1} \bar y_{k+1} + o(1),
\end{eqnarray*}
where we have set $a_{k}=0$ whenever $k$ is odd. Then, the sequence $( a_{\frac{k}2+1})_{k\geq 1}$ satisfies (\ref{eq:Hamburger}), which gives the proposition.
\\
\\
\par{\bf Acknowledgment:}
\\The author would like to gratefully thank Alice Guionnet, Florent Benaych-Georges, Mikael de la Salle, Djalil Chafa\"i and Charles Bordenave for useful discussions. He also acknowledges the financial support of the ANR GranMa.
\\
\\This paper has known significant progress during some travels in 2011/12: two visits of Ashkan Nikeghbali and Kim Dang at the Institut f\"ur Mathematik of Zurich, the school ''Vicious walkers and random matrices`` in les Houches organized by the CNRS, and the summer school ''Random matrix theory and its applications to high-dimensional statistics`` in Changchun, founded jointly by the CNRS of France and the NSF of China. The author gratefully acknowledges the organizers of these events for providing an inspiring working environment.

\bibliographystyle{plain} 
\bibliography{biblio}

\def\cprime{$'$}
\begin{thebibliography}{10}

\bibitem{AGZ}
G.~W. Anderson, A.~Guionnet, and O.~Zeitouni.
\newblock {\em An Introduction to Random Matrices}, volume 118 of {\em
  Cambridge studies in advanced mathematics}.
\newblock Cambridge University Press, 2010.

\bibitem{BBGS11}
R.~{Basu}, A.~{Bose}, S.~{Ganguly}, and R.~{Subhra Hazra}.
\newblock {Joint convergence of several copies of different patterned random
  matrices}.
\newblock {\em ArXiv e-prints}, August 2011.

\bibitem{BDG}
S.~Belinschi, A.~Dembo, and A.~Guionnet.
\newblock Spectral measure of heavy tailed band and covariance random matrices.
\newblock {\em Comm. Math. Phys.}, 289(3):1023--1055, 2009.

\bibitem{BG}
G.~Ben~Arous and A.~Guionnet.
\newblock The spectrum of heavy tailed random matrices.
\newblock {\em Comm. Math. Phys.}, 278(3):715--751, 2008.

\bibitem{BGCD12}
F.~Benaych-Georges and T.~Cabanal-Duvillard.
\newblock Marchenko-pastur theorem and bercovici-pata bijections for
  heavy-tailed or localized vectors.
\newblock arXiv:1204.5154v3, preprint, http://arxiv.org/abs/1204.5154.

\bibitem{BCC2}
C.~Bordenave, P.~Caputo, and D.~Chafa\"i.
\newblock Spectrum of large random reversible markov chains: heavy tailed
  weights on the complete graph.
\newblock {\em Ann. Probab.}, 39:1544--1590, 2011.

\bibitem{CC}
M.~Capitaine and M.~Casalis.
\newblock Asymptotic freeness by generalized moments for {G}aussian and
  {W}ishart matrices. {A}pplication to beta random matrices.
\newblock {\em Indiana Univ. Math. J.}, 53(2):397--431, 2004.

\bibitem{CHU97}
F.~R.~K. Chung.
\newblock {\em Spectral graph theory}, volume~92 of {\em CBMS Regional
  Conference Series in Mathematics}.
\newblock Published for the Conference Board of the Mathematical Sciences,
  Washington, DC, 1997.

\bibitem{CDS95}
D.~M. Cvetkovi{\'c}, M.~Doob, and H.~Sachs.
\newblock {\em Spectra of graphs}.
\newblock Johann Ambrosius Barth, Heidelberg, third edition, 1995.
\newblock Theory and applications.

\bibitem{DJ10}
X.~Ding and T.~Jiang.
\newblock Spectral distributions of adjacency and {L}aplacian matrices of
  random graphs.
\newblock {\em Ann. Appl. Probab.}, 20(6):2086--2117, 2010.

\bibitem{DYK}
K.~Dykema.
\newblock On certain free product factors via an extended matrix model.
\newblock {\em J. Funct. Anal.}, 112(1):31--60, 1993.

\bibitem{GUI}
A.~Guionnet.
\newblock {\em Large random matrices: lectures on macroscopic asymptotics},
  volume 1957 of {\em Lecture Notes in Mathematics}.
\newblock Springer-Verlag, Berlin, 2009.
\newblock Lectures from the 36th Probability Summer School held in Saint-Flour,
  2006.

\bibitem{HAM}
H.~Hamburger.
\newblock \"{U}ber eine {E}rweiterung des {S}tieltjesschen {M}omentenproblems.
\newblock {\em Math. Ann.}, 82(3-4):168--187, 1921.

\bibitem{KSV04}
O.~Khorunzhy, M.~Shcherbina, and V.~Vengerovsky.
\newblock Eigenvalue distribution of large weighted random graphs.
\newblock {\em J. Math. Phys.}, 45(4):1648--1672, 2004.

\bibitem{Lov09}
L.~Lov{\'a}sz.
\newblock Very large graphs.
\newblock In {\em Current developments in mathematics, 2008}, pages 67--128.
  Int. Press, Somerville, MA, 2009.

\bibitem{MAL12}
c.~Male.
\newblock The distribution of traffics and their free product.
\newblock arXiv:1111.4662v3, preprint, http://arxiv.org/abs/1111.4662.

\bibitem{MAL}
C.~Male.
\newblock The norm of polynomials in large random and deterministic matrices.
\newblock {\em Probab. Theory Related Fields}, 2011.

\bibitem{MS09}
J.~A. Mingo and R.~Speicher.
\newblock Sharp bounds for sums associated to graphs of matrices.
\newblock {\em J.F.A.}, 262:Issue 5, p. 2272Ð2288, 2012.

\bibitem{NEA}
M.~Neagu.
\newblock Asymptotic freeness of random permutation matrices from gaussian
  matrices.
\newblock {\em J. Ramanujan Math. Soc.}, 20(3):189--213, 2005.

\bibitem{NIC93}
A.~Nica.
\newblock Asymptotically free families of random unitaries in symmetric groups.
\newblock {\em Pacific J. Math.}, 157(2):295--310, 1993.

\bibitem{NS}
A.~Nica and R.~Speicher.
\newblock {\em Lectures on the combinatorics of free probability}, volume 335
  of {\em London Mathematical Society Lecture Note Series}.
\newblock Cambridge University Press, Cambridge, 2006.

\bibitem{RYA}
{\O}.~Ryan.
\newblock On the limit distributions of random matrices with independent or
  free entries.
\newblock {\em Comm. Math. Phys.}, 193(3):595--626, 1998.

\bibitem{ST10}
Mariya Shcherbina and Brunello Tirozzi.
\newblock Central limit theorem for fluctuations of linear eigenvalue
  statistics of large random graphs.
\newblock {\em J. Math. Phys.}, 51(2):023523, 20, 2010.

\bibitem{Shl96}
D.~Shlyakhtenko.
\newblock Random {G}aussian band matrices and freeness with amalgamation.
\newblock {\em Internat. Math. Res. Notices}, (20):1013--1025, 1996.

\bibitem{ZAK}
I.~Zakharevich.
\newblock A generalization of {W}igner's law.
\newblock {\em Comm. Math. Phys.}, 268(2):403--414, 2006.

\end{thebibliography}

\end{document}